\documentclass[11pt,a4paper]{amsart}
\usepackage{amssymb,amsmath}
\usepackage{booktabs}
\usepackage{enumerate}
\usepackage{hyperref}
\usepackage{graphicx}
\usepackage{subfigure}

\usepackage{subfigure}
\renewcommand{\thesubfigure}{\alph{subfigure}}
\makeatletter
\renewcommand{\@thesubfigure}{(\thesubfigure)\space}
\renewcommand{\p@subfigure}{\thefigure}
\makeatother

\newtheorem{theorem}{Theorem} \theoremstyle{definition}
\newtheorem{definition}{Definition}
\newtheorem{lemma}[theorem]{Lemma}
\newtheorem{proposition}[theorem]{Proposition}

\theoremstyle{remark}
\newtheorem{remark}{Remark}
\numberwithin{remark}{section}
\numberwithin{theorem}{section}
\numberwithin{equation}{section}
\numberwithin{definition}{section}
\newtheorem{assumption}{Assumption}
\numberwithin{assumption}{section}

\newcommand{\tri}{\mathcal{T}}
\newcommand{\triH}{\tri_H}
\newcommand{\trih}{\tri_h}
\newcommand{\grad}{\nabla}
\newcommand{\ccont}{\tfrac{\beta}{\alpha}}
\newcommand{\dx}{\operatorname*{d}\hspace{-0.3ex}x}

\newcommand{\ddiv}{\operatorname*{div}}
\newcommand{\diam}{\operatorname*{diam}}

\newcommand{\support}{\operatorname*{supp}}
\newcommand{\dimension}{\operatorname*{dim}}

\newcommand{\N}{\mathcal{N}}
\newcommand{\R}{\mathbb{R}}
\newcommand{\QI}{\mathcal{I}_H}
\newcommand{\Cint}{C_{\operatorname*{qip}}}
\newcommand{\Vf}{V^{\operatorname*{fs}}}
\newcommand{\Pf}{\mathcal{P}^{\operatorname*{fs}}}

\newcommand{\Vc}{V^{\operatorname*{cs}}}

\newcommand{\uc}{u^{\operatorname*{cs}}}
\newcommand{\G}{{\mathcal G}}
\newcommand{\E}{{\mathcal E}}

\title[Numerical Upscaling at High Contrast]{Robust Numerical Upscaling of Elliptic Multiscale Problems at High Contrast}

\author{Daniel Peterseim$^*$ and Robert Scheichl}
\address[D. Peterseim]{Institut f\"ur Numerische Simulation der Universit\"at Bonn, Wegelerstr. 6, 53115 Bonn, Germany}
\email{peterseim@ins.uni-bonn.de}
\address[R. Scheichl]{Department of Mathematical Sciences, University
  of Bath, Claverton Down, Bath BA2 7AY, UK}
\email{r.scheichl@bath.ac.uk}

\date{\today}

\keywords{finite element, multiscale, upscaling, computational
  homogenization, high contrast}

\subjclass[2000]{65N30, 65N25, 65N15}

\begin{document}

\begin{abstract}
We present a new approach to the numerical upscaling for elliptic problems with rough diffusion coefficient at high contrast. It is based on the localizable orthogonal decomposition of $H^1$ into the image and the kernel of some novel stable quasi-interpolation operators with local $L^2$--approximation properties, independent of the contrast. We identify a set of sufficient assumptions on these quasi-interpolation operators that guarantee in principle optimal convergence without pre-asymptotic effects for high-contrast coefficients. We then give an example of a suitable operator and establish the assumptions for a particular class of high-contrast coefficients. So far this is not possible without any pre-asymptotic effects, but the optimal convergence is independent of the contrast and the asymptotic range is largely improved over other discretisation schemes. The new framework is sufficiently flexible to allow also for other choices of quasi-interpolation operators and the potential for fully robust numerical upscaling at high contrast.
\end{abstract}

\maketitle

\footnotetext{Supported by the Sino-German Science Center on the occasion of the Chinese-German Workshop on Computational and Applied Mathematics in Augsburg 2015.}

\section{Introduction}
This paper presents and analyses a novel numerical upscaling technique for the approximate solution of a prototypical partial differential equation with arbitrary positive bounded coefficients. 
The focus is on coefficients $A$ that are strongly heterogeneous, i.e., $A$ may vary rapidly on several non-separated scales and, moreover, the physical contrast (the ratio between global upper and lower bounds of its spectrum) may be very large. 

The precise setting of the paper is as follows. Let $\Omega\subset\mathbb{R}^{d}$ be a bounded polyhedral domain and let $A\in L^\infty(\Omega,\R^{d\times d}_{\operatorname*{sym}})$ be a matrix-valued coefficient with uniform spectral bounds $0<\alpha\leq \beta<\infty$,
\begin{equation}\label{e:spectralbound}
\sigma(A(x))\subset [\alpha,\beta],
\end{equation}
for almost all $x\in \Omega$.
Given some forcing term $g\in L^2(\Omega)$, we want to approximate the unknown weak solution $u$ of the linear elliptic partial differential equation $-\ddiv(A\grad u)=g$ with homogeneous Dirichlet boundary condition. The function $u\in V:=H_{0}^{1}( \Omega)$ is uniquely characterized by the variational problem
\begin{equation}\label{e:model}
b(u,v):=\int_\Omega (A\nabla u)\cdot \nabla v\dx = \int_\Omega g v\dx=:G(v),\quad \text{for all} \ \ v\in V.
\end{equation}
The accuracy of standard Galerkin finite element approximations of the unknown function $u$ depends crucially on the regularity of the underlying data. On the one hand, the rate of convergence under mesh refinement depends on interior angles of the domain and differentiability properties of $A$. On the other hand, even if the data is sufficiently regular so that a certain rate of convergence is possible, it may be observed only if the width $h$ of the underlying mesh is sufficiently small. In this context, the notion ``sufficiently small'' depends on data oscillations and the contrast in a critical way. E.g., for a scalar coefficient $A$ that oscillates between $\alpha$ and $\beta$ at some frequency $\varepsilon^{-1}$ for some small parameter $\varepsilon$, the asymptotic rate of convergence is not observed unless $h\lesssim \varepsilon$. In addition, even $h\lesssim (\ccont)^{-1}\varepsilon$ is necessary to decrease the energy error below $100\%$, which is too restrictive in many interesting cases. We emphasize that this condition is sharp for practically relevant right-hand sides $g$. 

We are therefore dealing with pre-asymptotic effects for standard finite element methods and other related schemes such as finite volumes or finite differences. Due to the high variability of the coefficient functions, one requires extremely fine computational grids that are able to capture all the fine scale oscillations and discontinuities. Hence, the numerical treatment of such equations is expensive in the sense that standard approaches result in systems of equations of enormous size and, hence, in a tremendous computational demand that can not be handled in a lot of scenarios.

This paper presents a new approach for numerical upscaling based on localizable orthogonal decompositions (LOD) into a low-dimensional coarse space (where we are looking for our approximation) and a high-dimensional remainder space. Some selectable quasi-interpolation operator serves as the basis of the decompositions. The coarse space is spanned by computable basis functions with local support. The basic methodology was recently introduced in \cite{2011arXiv1110.0692M} and generalized in \cite{Elfverson.Georgoulis.Mlqvist.ea:2012,2012arXiv1211.3551H,2012arXiv1211.5954H,2013arXiv1312.5922P}. For moderate contrast and arbitrary oscillatory coefficients this methodology yields approximations that converge to the true solution at the optimal rate (with respect to the coarse mesh size) without any pre-asymptotic effects. The analysis avoids the strong assumptions usually made in the classical homogenization framework, such as periodicity or scale separation. 

The promising numerical results in \cite{2011arXiv1110.0692M,Elfverson.Georgoulis.Mlqvist.ea:2012,HM_arXiv2013} for high-contrast model coefficients are not yet  reflected by the theoretical results for localized bases in those references, because the physical contrast $\beta/\alpha$ appears to be a critical parameter. The dependence on $\beta/\alpha$ enters the error analysis via norm equivalences
\begin{equation}\label{e:norms}
\begin{aligned}
  \beta^{-1/2}\|A^{1/2}\nabla \cdot\|_{L^2}&\leq\|\nabla \cdot\|_{L^2}\leq \alpha^{-1/2}\|A^{1/2}\nabla \cdot\|_{L^2},\\ \beta^{-1/2}\|A^{1/2} \cdot\|_{L^2}&\leq\| \cdot\|_{L^2}\leq \alpha^{-1/2}\|A^{1/2} \cdot\|_{L^2}.
\end{aligned}
\end{equation}
These equivalences are heavily used to connect variational techniques such as Galerkin orthogonality with  approximation properties of standard quasi-interpolation operators in standard coefficient-independent Sobolev spaces. The idea of this paper is to circumvent the critical norm equivalences by using coefficient-dependent quasi-interpolation operators, e.g. in \cite{MR2861254}, which enjoy optimal approximation properties in $A$-weighted Sobolev spaces.

Our multiscale method is fully defined by the choice of the quasi-interpolation operator $\QI$. We state a sufficient set of conditions on $\QI$ that will yield approximations $\uc$ that converge linearly to $u$ in the energy norm with respect to the coarse mesh size $H$, without any pre-asymptotic effects and independent of the contrast. More precisely, we show that local pre-computations of the coarse basis functions on vertex patches of diameter $\approx H\log(H^{-1}\sqrt{\beta/\alpha})$ suffice to derive the following error bound
\begin{equation*}
 \|A^{1/2} \nabla(u-\uc) \|_{L^2(\Omega)} \leq C H.
\end{equation*}
Here, $C$ denotes a generic constant that is independent of the computational grid and depends only on the constants in the abstract assumptions that we have made on $\QI$. In particular, if $\QI$ can be chosen such that all the assumptions hold with constants that are independent of contrast and fine scale heterogeneity then the convergence is also independent of such pre-asymptotic effects. 

Employing (as an example) novel quasi-interpolation techniques related to those analysed in \cite{MR2861254} we are indeed able to satisfy the sufficient conditions with constants that are independent of the contrast. So far this is only possible under some conditions on the geometry of the coefficient relative to the coarse grid. Moreover, the constant $C$ is not independent of $H/\varepsilon$ and the method is thus not without pre-asymptotic effects, but it extends the asymptotic regime far beyond that of other methods independently of the contrast. Despite these limitations, this result is the first one beyond heuristics to show that numerical upscaling for certain classes of high-contrast problems is possible. It may pave the way towards a comprehensive understanding of general high-contrast coefficients. In fact, our numerical tests do not show any strong pre-asymptotic effects. We shall emphasize at this point that in the LOD framework the coarse basis functions depend on the particular choice of the quasi-interpolation operator. In that sense, the methods analysed in this paper differ from those presented in \cite{2011arXiv1110.0692M,Elfverson.Georgoulis.Mlqvist.ea:2012,HM_arXiv2013}. Our new theoretical results improve the dependence on $\ccont$ of the convergence rate and of the scaling of the supports of the underlying basis functions in those papers, as well as in the alternative approaches in the literature \cite{BabLip11,MR2721592,OwhadiZ11,MZA:9015370,KY15}. The new results apply also to a more general class of coefficients than the analysis in \cite{MR2684351} and \cite{Peterseim:2013} which is also independent of $\ccont$.
 
Our approach belongs to the large class of multiscale methods. These methods, typically, decouple the necessary fine scale computations into local parts to decrease the computational cost without suffering from a remarkable loss in accuracy. Prominent examples of multiscale methods are the Multiscale Finite Element Method (MsFEM) proposed by Hou and Wu \cite{MR1455261} and the Heterogeneous Multiscale Method (HMM) by E and Engquist \cite{MR1979846}. In contrast to our approach, MsFEM and HMM are typically not constructed for a direct approximation of the unknown solutions but for homogenized solutions and corresponding correctors instead. Thus, the reliable approximation of the exact solution is up to unknown modeling errors that punishes the lack of proper periodicity and scale separation. Our framework is related to another classical multiscale method, the Variational Multiscale Method (VMM) proposed by Hughes et al. \cite{MR1660141} (see also \cite{MR2300286,Peterseim:2015}). In contrast to MsFEM and HMM, the VMM aims at a direct approximation of the exact solution without suffering from a modeling error remainder arising from homogenization theory. For connections between the methodologies we refer to \cite{2012arXiv1211.5954H,Peterseim:2015}. An interesting extension of the MsFEM method to more general heterogeneous coefficients without assumptions like periodicity and scale separation is the Generalised MsFEM \cite{Efendiev2013116}.

The remaining part of the paper is structured as follows. Section~\ref{s:abstract} defines the abstract methodological framework. In particular, abstract axioms on the underlying quasi interpolation are formulated that guarantee contrast-independent performance of the corresponding method shown in Section~\ref{s:apriori}. In Section~\ref{s:qi} we then present particular examples of quasi interpolation operators that satisfy the previous axioms for certain classes of coefficients. Section~\ref{s:numexp} discusses the results and their limitations in the light of several numerical experiments.

\section{An abstract multiscale method}\label{s:abstract}
In this section, we propose an abstract multiscale method based on the framework of localizable othogonal decompositions. The framework is inspired by the Variational Multiscale Method of Hughes et al. \cite{MR1660141} but takes a very different point of view and follows the specific constructions proposed in \cite{2011arXiv1110.0692M,2012arXiv1211.3551H,2012arXiv1211.5954H}. For a re-interpretation within the stabilization framework of the original VMM see \cite{Peterseim:2015}.

The key ingredient is a continuous, surjective and uniformly stable quasi-interpolation operator from some fine scale finite element space to an initial coarse space 
$\tilde{V}^{\text{cs}}$ that has certain $L^2$-approximation properties uniformly with respect to the contrast. Following the approach in \cite{2011arXiv1110.0692M}, it is then possible to design a new coarse space that is provably robust even in the high contrast regime. The localization of the basis functions depends only mildly on the contrast.

\subsection{Standard finite element discretization}\label{ss:classical}
Let $\triH$ denote a regular triangulation of $\Omega$ into closed simplices and let $H:\overline\Omega\rightarrow\mathbb{R}_{>0}$ denote the $\triH$-piecewise constant mesh size function with $H\vert_T=H_T:=\diam(T)$ for all $T\in\triH$. Additionally, let $\trih$ be a regular triangulation of $\Omega$ that is supposed to be a refinement of $\triH$. We assume that $\trih$ is sufficiently small so that all fine scale features of the coefficient $A$ are captured. The mesh size $h$ denotes the maximum diameter of an element of $\trih$.
The corresponding classical (conforming) finite element spaces of continuous piecewise polynomials of degree $1$ are given by
\begin{align*}
V_H&:=\{v_H\in H^1_0(\Omega)\;\vert\;\forall T\in\triH:\enspace (v_H)\vert_T \text{ is affine}\},\\
V_h&:=\{v_h\in H^1_0(\Omega)\;\vert\;\forall K\in\mathcal{T}_h: \enspace (v_h)\vert_K \text{ is affine}\}.
\end{align*}
By $\N_H$ we denote the set of interior vertices of $\triH$ (representing the degrees of freedom of the coarse finite element spaces). For every vertex $z \in\N_H$, let $\lambda_z\in V_H$ denote the associated nodal basis function (hat function) characterized by the property $\lambda_y(z)=\delta_{yz}$ for all $y,z\in \N_H$. We will also need the vertex patches 
\begin{equation}
\omega_z := \text{supp} \lambda_z = \text{int}\left(\cup\left\{T\in\tri_H\;|\;x\in T\right\}\right).
\end{equation} 

From now on, we denote by $u_h \in V_h$ the classical finite element (FE) approximation of $u$ in the discrete (highly resolved) space $V_h$, i.e., $u_h \in V_h$ solves
\begin{align}
\label{e:modelweak}b(u_h,v_h)= G(v_h),\quad\text{for all } v_h \in V_h.
\end{align}
We assume that $V_h$ resolves the micro structure, i.e., that the error ${\| u - u_h \|_{H^1(\Omega)}}$ becomes sufficiently small by falling below a given tolerance.
Moreover, we assume that the contrast relative to the fine mesh $\tri_h$ is small in the sense of
\begin{equation}\label{e:spectralboundh}
\underset{x\in \tau}{\operatorname{ess}\sup}
\sup\limits_{v\in\mathbb{R}^{d}\setminus\{0\}}\dfrac{(A( x) v)\cdot v
}{v\cdot v}\lesssim \underset{x\in \tau}{\operatorname{ess}\inf}
\inf\limits_{v\in\mathbb{R}^{d}\setminus\{0\}}\dfrac{(A( x) v)\cdot v
}{v\cdot v},
\end{equation}
for all $\tau\in \trih$. 

\subsection{Abstract quasi-interpolation}\label{ss:intpolabstract}

As stated above, the key tools in our construction are an initial
coarse space $\tilde{V}^{\text{cs}} \subset V_h$ with certain local $L^2$-approximation properties and a quasi-interpolation operator $\QI: V_h\rightarrow \tilde{V}^{\text{cs}}$ that is linear, continuous and surjective. The kernel of this operator is going to be our fine space (or remainder space) $\Vf_h$. 

To simplify the presentation we will for the most part only consider the special case, when the piecewise linear coarse space $V_H$ has the appropriate $L^2$--approximation properties. As we will see, this allows us to treat a very interesting class of highly varying coefficients, namely those that are locally quasi-monotone (in the sense of \cite{PS_IMAJNA2012}). We will comment briefly in Remark~\ref{rem:general1} below on how the framework can be extended also to other initial coarse spaces and to more general highly varying coefficients. 

Thus, from now on we set $\tilde{V}^{\text{cs}} := V_H$ and characterize the interpolation operator via some set of assumptions that must be fulfilled in order to derive a contrast-independent convergence result for the constructed multiscale method. Specific constructions are given in Section~\ref{s:qi}.

\begin{assumption}[Assumptions on the interpolation]\label{a:QI}
We make the following assumptions on the interpolation operator $\QI: V_h\rightarrow V_H$:
\begin{itemize}
\item[(QI1)] $\QI \in L(V_h,V_H)$ is linear and continuous,
\item[(QI2)] the restriction of $\QI$ to $V_H$ is an isomorphism,
\item[(QI3)] there exists a generic constant $\Cint$, such
  that for all $v_h\in V_h$ and for all $T\in\triH$,
\begin{equation*}
H_T^{-1}\|A^{1/2}(v_h-\QI v_h)\|_{L^2(T)}+\|A^{1/2}\nabla(v_h-\QI v_h)\|_{L^2(T)}\leq \Cint \|A^{1/2} \nabla v_h\|_{L^2(\omega_T)}
\end{equation*}
with $\omega_{T} := \text{int}\left({\bigcup \{K \in \triH| \hspace{2pt} K \cap T \neq \emptyset } \}\right)$.
\item[(QI4)] there exists a generic constant $\Cint'$, such
  that for all $v_H \in V_H$ there exists $v_h \in V_h$ with the properties
\begin{align*}
\QI v_h = v_H, \quad & \text{supp} \, v_h \subset \text{supp} \, v_H
\quad \text{and}\\
\|A^{1/2}\nabla v_h\|_{L^2(\Omega)} \le & \ \Cint'
 \|A^{1/2}\nabla v_H\|_{L^2(\Omega)}. 
\end{align*} 
\end{itemize}
\end{assumption}
Some remarks are in order to explain the assumptions (QI1)-(QI4). Linearity and continuity (QI1) as well as invertibility on the finite element space (QI2) are minimal assumptions that are typically satisfied by Cl\'ement-type operators. Note that $\QI$ does not need to be a projection onto the finite element space $V_H$. The conditions ensure that the concatenation $(\QI\vert_{V_H})^{-1}\QI$ always defines such a projection. 
Condition (QI3) yields the crucial local approximation (resp. stability) properties in weighted $L^2$ (resp. energy) norm in $\tilde{V}^{\text{cs}}$. 
Finally, assumption (QI4) ensures that any coarse finite element function $v_H\in V_H$ is the image of some function $v_h\in V_h$ under $\QI$ with smaller or equal support. In other words, there exists some bounded left inverse of $\QI$ that preserves local supports. This property also compensates the possible lack of a projection property. If $\QI$ was a projection then (QI4) would be satisfied by choosing $v_h=v_H$.

\begin{remark}
\label{rem:general1}
More generally, the initial coarse space $\tilde{V}^{\text{cs}}$ could be any subspace of $V_h$ that admits a local basis $\{\tilde{\lambda}_{z,\ell} \in V_h : z\in \N_H \ \text{and} \ \ell=1,\ldots,L_z\}$ with (i) $L_z \ge 1$ basis functions associated with each vertex $z \in \N_H$, (ii) $\text{supp}(\tilde{\lambda}_{z,\ell}) \subset \omega_z$, (iii) $\| \tilde{\lambda}_{z,\ell}\|_{L^\infty(\Omega)} \lesssim 1$, and possibly further conditions such as a partition of unity property; see also \cite{2013arXiv1312.5922P}. Typical examples in the context of high contrast would be standard or generalised multiscale finite element functions \cite{MR1455261,Efendiev2013116} and the associated natural quasi-interpolation operators \cite{MR2861254}. The natural $L^2$-norm in (QI3) will often also be different in those cases. 
\end{remark}

\subsection{Two-scale orthogonal decomposition and global coarse space}\label{ss:multiscale}
In this section, we construct a decomposition of 
the high resolution finite element space $V_h$
into a low-dimensional space $\Vc$ and some high-dimensional remainder space $\Vf$. 
As subspaces of $V_h$, $\Vc$ and $\Vf$ depend on the fine scale discretization parameter $h$. Since the choice of $h$ is not the topic of this paper, this dependence will not be reflected by our notation. Note that the subsequent derivation remains valid in the limit $h\rightarrow 0$ (cf. \cite{2011arXiv1110.0692M,MP12}).

Let $\QI: V_h\rightarrow V_H$ denote an interpolation operator that satisfies the properties (QI1)-(QI2) from Assumption~\ref{a:QI}.
We define $\Vf$ as the kernel of $\QI$ in $V_h$,
\begin{equation*}\label{e:finescale}
 \Vf:=\{v\in V_h\;\vert\;\QI v=0\}.
\end{equation*}
The space $\Vf$ represents the finescale features in $V_h$ not captured by $V_H$. This definition along with properties (QI1) and (QI2) give rise to the decomposition $V_h = V_H \oplus \Vf$. 

The key step towards the definition of an appropriate coarse space is to orthogonalize this decomposition with respect to the scalar product $b(\cdot,\cdot)=( A \nabla \cdot, \nabla \cdot)_{L^2(\Omega)}$ induced by the problem. For this purpose, we define a corresponding $a$-orthogonal projection $\Pf: V_h \rightarrow \Vf$ as follows. Given $v\in V_h$, define $\Pf(v)\in\Vf$ as the unique solution of
\begin{equation*}\label{e:finescaleproj}
 b(\Pf(v),w)=b(v,w),\quad\text{for all }w \in \Vf.
\end{equation*}

The coarse scale space is defined by
$$\Vc:=(1-\Pf)V_H$$ 
and yields the orthogonal splitting
\begin{align}
\label{splitting-2} V_h = \Vc \oplus \Vf\quad\text{with}\quad b(\Vc,\Vf)=0.
\end{align}

We shall introduce a basis of $\Vc$. The image of the nodal basis function $\lambda_z \in V_H$ under the fine scale projection $\Pf$ is denoted by $\phi_z=\Pf(\lambda_z)\in \Vf$, i.e., $\phi_z$ satisfies the corrector problem
\begin{equation}\label{e:corrector}
 b(\phi_z,w)=b(\lambda_z,w),\quad\text{for all }w\in \Vf.
\end{equation}
A basis of $\Vc$ is then given by the modified nodal basis
\begin{equation}\label{e:basiscoarse}
 \{\psi_z:= \lambda_z-\phi_z\;\vert\ z\in\N_H\}.
\end{equation}

\begin{definition}[Global coarse approximation]\label{definition-full-vmm-no-truncation}
The Galerkin approximation $\uc \in \Vc$ of the exact weak solution $u$ of \eqref{e:model} and of the FE reference solution $u_h$ of \eqref{e:modelweak} is defined as the solution of
\begin{equation}
\label{e:modeldiscreteideal} b(\uc,v) = G(v), \quad\text{ for all }v\in\Vc.
\end{equation}
\end{definition}

In general, the basis functions $\psi_z$ have global support $\Omega$ and their pre-computation involves one fine scale computation on the whole domain $\Omega$ per coarse degree of freedom. In this sense, the pre-computation of this basis is expensive and the corresponding Galerkin discretization \eqref{e:modeldiscreteideal} yields small but densely populated stiffness and mass matrices. In certain situations, it may still be a reasonable coarsening (see Section~\ref{ss:errorglobal}). 

A local basis may be achieved by localization of the corrector problems. Since the right-hand side of \eqref{e:corrector} induced by $\lambda_z$ has small support, the correctors $\phi_z$ show an exponential decay outside of the support of $\lambda_z$. Hence, we are able to localize the correctors and their computation to local subdomains in Section~\ref{ss:feasible}. 

\subsection{Localized coarse space}\label{ss:feasible}
We approximate the global coarse space $\Vc$ from the previous section by truncating the corrector problems \eqref{e:corrector} for the basis functions to local patches of coarse elements as suggested in \cite{2011arXiv1110.0692M}.

Let $k\in\mathbb{N}$ be a discretization parameter that reflects the localization of the finescale computations. Define nodal patches of $k$-th order $\omega_{z,k}$ about $z\in\N_H$ by
\begin{equation}\label{e:omega}
 \begin{aligned}
 \omega_{z,1}&:=\support \lambda_z=\cup\left\{T\in\tri_H\;|\;x\in T\right\},\\
 \omega_{z,k}&:=\cup\left\{T\in\tri_H\;|\;T\cap {\omega}_{z,{k-1}}\neq\emptyset\right\}\quad k=2,3,4\ldots .
\end{aligned}
\end{equation}
Define localized finescale spaces 
$$\Vf(\omega_{z,k}):=\{v\in\Vf\;\vert\;v\vert_{\Omega\setminus\omega_{z,k}}=0\},\quad z\in\N_H,$$
by intersecting $\Vf$ with those functions that vanish outside the patch $\omega_{z,k}$.
\begin{definition}[Local correctors]\label{local-corrector}
Local correctors $\phi_{z,k}\in \Vf(\omega_{z,k})$ are unique solutions of
\begin{equation}\label{e:Tnodallocal}
 b(\phi_{z,k},w)=b(\lambda_z,w),\quad\text{for all }w\in \Vf(\omega_{z,k}).
\end{equation}
\end{definition}
The local correctors $\phi_{z,k}\in \Vf(\omega_{z,k})$ are approximations of the global correctors $\phi_z\in\Vf$ from \eqref{e:corrector} with local support $\omega_{z,k}$. Note that homogeneous Dirichlet boundary condition are enforced on $\partial\omega_{z,k}$. We define localized coarse spaces
\begin{equation*}
\Vc_k=\operatorname*{span}\{\psi_{z,k}:=\lambda_z-\phi_{z,k}\;\vert\;x\in\N_H\}\subset V.
\end{equation*}
\begin{definition}[Local coarse approximation]\label{definition-full-vmm-truncation}
Given some localization parameter $k\in\mathbb{N}$, the Galerkin approximation of \eqref{e:model} and \eqref{e:modelweak} reads: find $\uc_k\in\Vc_k$ such that
\begin{equation}\label{e:modeldiscretek}
 b(\uc_k,v) =  G(v), \quad\text{ for all }v\in\Vc_k.
\end{equation}
\end{definition}
Note that $\dimension\Vc_k=|\N_H|=\dimension V_H$, that is, the number of degrees of freedom of the proposed method \eqref{e:modeldiscreteideal} is the same as for the classical finite element method on the coarse mesh $\tri_H$, or more generally, the same as for the initial coarse space $\tilde{V}^{\text{cs}}$. The basis functions of the multiscale method have local support. The overlap is proportional to the parameter $k$. The error analysis of Section~\ref{ss:decay} shows that the choice $k\approx2\log(H^{-1})+{1}/{2}\log(\beta/\alpha)$ suffices to preserve the desired linear convergence in $H$.

\subsection{Alternative localization techniques}\label{ss:locmod}
A modified technique for localization is presented in \cite{2012arXiv1211.5954H}. Define element patches of $k$-th order $\omega_{T,k}$ about $T\in\tri_H$ by
\begin{equation*}
 \begin{aligned}
 \omega_{T,1}&:=T,\\
 \omega_{T,k}&:=\cup\left\{T\in\tri_H\;|\;T\cap {\omega}_{T,{k-1}}\neq\emptyset\right\},\quad k=2,3,4\ldots .
\end{aligned}
\end{equation*}
Define localized finescale spaces 
$$\Vf(\omega_{T,k}):=\{v\in\Vf\;\vert\;v\vert_{\Omega\setminus\omega_{T,k}}=0\},\quad T\in\tri_H,$$
by intersecting $\Vf$ with those functions that vanish outside the patch $\omega_{T,k}$.
The corrections are then computed in a two-step procedure. 
First, for any element $T\in\triH$ and for any $y\in\N_H(T):=\N_H\cap T$, compute
$\tilde{\psi}_{T,y,k}\in \Vf(\omega_{T,k})$ as the unique solution of
\begin{equation}\label{e:localprobs}
 b(\tilde{\psi}_{T,y,k},w)=(A\nabla \lambda_y,\nabla w)_{L^2(T)}\quad\text{for all }w\in \Vf(\omega_{T,k}).
\end{equation}
For any node $z\in\N_H$, the corrector $\phi_{z,k}$ of $\lambda_z$ is then defined by
\begin{equation}\label{e:localprobs2}
 \tilde{\psi}_{z,k}:=\sum_{y\in \N_H\cap \omega_{z,1}}\tilde{\psi}_{T,y,k}.
\end{equation}
The local problems in \eqref{e:localprobs} are independent of each other and may be solved in parallel. However, in contrast to the localization of Section~\ref{ss:feasible}, the computation of the final correctors \eqref{e:localprobs2} requires communication among neighboring nodes. 
This two-step technique preserves the partition of unity property of the original basis in $V_H$ and, hence, yields slightly improved error bounds (cf. Remark~\ref{rem:hminus1}) when compared with the localization described above. The improved accuracy has also been observed in numerical experiments (cf. \cite{2012arXiv1211.5954H}). More general localization techniques with similar properties are discussed in \cite{2013arXiv1312.5922P}. Nevertheless, with regard to the already very technical error analysis of this paper, we will not include this improved localization strategy in our theory.
 
\section{Abstract a priori error analysis}\label{s:apriori}
This section studies the error of the coarse scale approximations of Definitions~\ref{definition-full-vmm-no-truncation} and \ref{definition-full-vmm-truncation} under the abstract assumptions (QI1)-(QI4) on the underlying quasi-interpolation operator $\QI$. 

Here and throughout this paper, \label{p:lesssim} the notation $a\lesssim b$ abbreviates $a\leq Cb$ with some multiplicative constant $C>0$ which only depends on the domain $\Omega$ and the shape regularity of underlying finite element meshes. We emphasize that $C$ does \emph{not} depend on discretization parameters and the coefficient $A$. Furthermore, $a\approx b$ abbreviates $a\lesssim b\lesssim a$.
For parameter-dependent inequalities, $a(\xi)\lesssim b(\xi)$ means that there exists some constant $C>0$ so that $a(\xi)\leq C b(\xi)$ holds for all parameters $\xi\in\Xi$, where the parameter set $\Xi$ will always be clear from the context.

\subsection{Error estimates for the global basis}\label{ss:errorglobal}
The following lemma shows the potential of the coarse space $\Vc$ and the corresponding coarse approximation $\uc$.
\begin{lemma}[Error of the global method]\label{l:ideal}
Let $u_h\in V$ solve \eqref{e:modelweak} and $\uc\in \Vc$ solve \eqref{e:modeldiscreteideal}. Under the condition (QI1)--(QI3), we have
\begin{equation*}
\|A^{1/2}\nabla(u_h-\uc)\|_{L^2(\Omega)}\lesssim \Cint\alpha^{-1/2}\|H g\|_{L^2(\Omega)}.
\end{equation*}
The estimate remains valid when $u_h$ is replaced with the weak solution $u\in V$ of \eqref{e:model}. 
\end{lemma}
\begin{proof}
The proof is almost verbatim the same as in \cite[Lemma 3]{2011arXiv1110.0692M}. The Galerkin orthogonality implies that the error $e:=u_h-\uc$ and the coarse space $\Vc$ are $b$-orthogonal. This shows that $e\in\Vf$ is a fine scale function and (QI3) proves, for any $T\in\triH$,
\begin{equation*}
\|e\|_{L^2(T)}= \|e- \QI e\|_{L^2(T)}\leq \alpha^{-1/2}\|A^{1/2}(e- \QI e)\|_{L^2(T)}
\leq \alpha^{-1/2}\Cint H\|A^{1/2}\nabla e\|_{L^2(\omega_T)}.
\end{equation*}
This, Galerkin orthogonality and the Cauchy-Schwarz inequality for sums then yield
\begin{multline*}
\|A^{1/2}\nabla e\|_{L^2(\Omega)}^2= b(u,u-\uc) = G(e)
\leq \sum_{T\in\triH}\|g\|_{L^2(T)}\|e\|_{L^2(T)}\\\lesssim \Cint\alpha^{-1/2}\|H g\|_{L^2(\Omega)}\|A^{1/2}\nabla e\|_{L^2(\Omega)},
\end{multline*}
where the constant hidden in the $\lesssim$ notation reflects the overlap of the element patches $\omega_T$.
\end{proof}
Note that the constant $\Cint$ appearing in the error bound of Lemma~\ref{l:ideal} can depend on the contrast if $\QI$ is not chosen properly; see also Section~\ref{ss:classicalQI} for a related discussion.

\subsection{Decay of global correctors}\label{ss:decay}
The following lemma is the key result of the paper.
\begin{lemma}[Decay of global correctors]\label{l:decay}
Let (QI1)--(QI4) be satisfied. For any node $z\in\N_H$ and any $k\in\mathbb{N}$, the correctors $\phi_z$ satisfy the estimate
$$\|A^{1/2}\nabla \phi_z\|_{L^2(\Omega\setminus\omega_{z,k})}\lesssim \exp\left(-\tfrac{k}{\Cint^2\Cint'}\right)\|A^{1/2}\nabla \phi_z\|_{L^2(\Omega)}$$
with constants $\Cint$, $\Cint'$ from Assumption~\ref{a:QI}. 
\end{lemma}
\begin{proof}
Let $z\in \N_H$ be arbitrary but fixed and, for the ease of notation, define $\phi:=\phi_z$ and  $\omega_{k}:=\omega_{z,k}$. 

For $j=1,\ldots,k-1$, define cut-off functions $\zeta_{k,j} :\Omega\rightarrow [0,1]\;\in W^{1,\infty}(\Omega)$ such that
\begin{subequations}\label{e:cutoff}
\end{subequations}
\begin{align}
(\zeta_{k,j})\vert_{\omega_{k-j}}&=0,\tag{\ref{e:cutoff}.a}\label{e:cutoff.a}\\
(\zeta_{k,j})\vert_{\Omega\setminus \omega_{k}}&=1,\;\text{and}\tag{\ref{e:cutoff}.b}\label{e:cutoff.b}\\
\forall T\in\triH,\;\|\nabla\zeta_{k,j}\|_{L^\infty(T)}&\lesssim (j H_T)^{-1}.\tag{\ref{e:cutoff}.c}\label{e:cutoff.c}
\end{align} Our particular choice of $\zeta_{k,j}$ is continuous and $\tri_H$-piecewise affine with nodal values
\begin{equation*}\label{e:cutoffH}
\begin{aligned}
 \zeta_{k,j}(y) &:= 0\quad\text{for all }y\in\N_H\cap \omega_{k-j},\\
 \zeta_{k,j}(y) &:= 1\quad\text{for all }y\in\N_H\cap \left(\Omega\setminus \omega_{k}\right),\text{ and}\\
 \zeta_{k,j}(y)&:= m/j\quad\text{for all }y\in\N_H\cap \partial\omega_{k-j+m},\;m=0,1,2,\ldots,j.
\end{aligned}
\end{equation*}

The cut-off function $\zeta_{k,j}$ allows one to estimate
\begin{multline}\label{e:decay0}
\|A^{1/2}\nabla \phi\|_{L^2(\Omega\setminus\omega_{k})}^2\leq (\zeta_{k,j} A\nabla \phi,\nabla \phi)_{L^2(\Omega\setminus\omega_{k-j})}\\= (A\nabla \phi,\nabla(\zeta_{k,j}
\phi))_{L^2(\Omega\setminus\omega_{k-j})}-\;\left(\phi A\nabla\phi,\nabla\zeta_{k,j}\right)_{L^2(\Omega\setminus\omega_{k-j})}.
\end{multline}

Let $I_h: V\cap C(\overline{\Omega})\rightarrow V_h$ denote the standard nodal interpolation operator with respect to the fine mesh $\tri_h$. 
According to (QI4) from Assumption~\ref{a:QI} there exists some $v\in \Vf(\Omega\setminus\omega_{k-j})$ such that $\QI(v)$ equals the coarse finite element function $\QI(I_h(\zeta_{k,j}\phi))\in V_H$. Introducing this into \eqref{e:decay0}, expanding and applying the Cauchy-Schwarz inequality yields
\begin{multline}\label{e:decay1}
\|A^{1/2}\nabla \phi\|_{L^2(\Omega\setminus\omega_{k})}^2\leq \|A^{1/2}\nabla \phi\|_{L^2(\Omega\setminus\omega_{k-j})} \|A^{1/2}\nabla(\zeta_{k,j}
\phi-I_h(\zeta_{k,j}
\phi)\|_{L^2(\Omega\setminus\omega_{k-j})}\\+|(A\nabla \phi,\nabla(I_h(\zeta_{k,j}
\phi)-v))_{L^2(\Omega\setminus\omega_{k-j})}|+\|A^{1/2}\nabla \phi\|_{L^2(\Omega\setminus\omega_{k-j})}\|A^{1/2}\nabla v\|_{L^2(\Omega\setminus\omega_{k-j})}\\+\|\phi A^{1/2}\nabla \zeta_{k,j}\|_{L^2(\Omega\setminus\omega_{k-j})}\|A^{1/2}\nabla\phi\|_{L^2(\Omega\setminus\omega_{k-j})}=:M_1+M_2+M_3+M_4.
\end{multline}
The four terms on the right-hand side of \eqref{e:decay1} are bounded separately as follows.

{\it Bound for $M_1$.} 
Recall the (local) approximation and stability properties of the nodal interpolation operator $I_h$ (in unweighted norms), i.e.
\begin{equation}\label{e:interr}
 \|\nabla (v-I_h v)\|_{L^2(t)}\lesssim h_t\|\nabla^2 v\|_{L^2(t)}\quad\text{and}\quad \|\nabla I_h v\|_{L^2(t)}\lesssim \|\nabla  v\|_{L^2(t)}
\end{equation}
for all polynomials $v$ on some element $t\in\tri_h$.  
Since $\zeta_{k,j}\phi$ is $\tri_h$-piecewise quadratic polynomial, this leads to 
\begin{multline*}\|A^{1/2}\nabla(\zeta_{k,j}\phi-I_h(\zeta_{k,j}\phi)\|_{L^2(\Omega)}^2\leq \sum_{t\in\tri_h}\|A\|_{L^\infty(t)}\|\nabla(\zeta_{k,j}\phi-I_h(\zeta_{k,j}\phi)\|_{L^2(t)}^2\\\lesssim \sum_{t\in\tri_h}\|A\|_{L^\infty(t)}h_t^2 \|\nabla^2(\zeta_{k,j}\phi)\|_{L^2(t)}^2\leq \sum_{t\in\tri_h}\|A\|_{L^\infty(t)}h_t^2 \|\nabla\zeta_{k,j}\cdot\nabla\phi\|_{L^2(t)}^2\\
\lesssim  j^{-2}\|A^{1/2}\nabla\phi\|_{L^2(\Omega\setminus\omega_{k-j})}^2,
\end{multline*}
where in the last step we used \eqref{e:cutoff},\eqref{e:spectralboundh}, and the trivial bound $h_t\leq H$. Thus, the bound for $M_1$ reads
$$M_1\lesssim j^{-1}\|A^{1/2}\nabla\phi\|_{L^2(\Omega\setminus\omega_{k-j})}^2.$$
{\it Bound for $M_2$.} The function $v$ was chosen so that $(I_h(\zeta_{k,j}
\phi)-v) \in \Vf$, which in conjunction with \eqref{e:corrector} implies that
$$M_2=|(A\nabla \lambda_z,\nabla (I_h(\zeta_{k,j}\phi)-v))_{L^2(\Omega)}|=0,$$ because the intersection of the supports of $\lambda_z$ and $I_h(\zeta_{k,j}\phi)-v$ has measure zero. \\
{\it Bound for $M_3$.} 
Due to (QI4), $v$ satisfies the estimate
$$\|A^{1/2}\nabla v\|_{L^2(\Omega\setminus\omega_{k-j})}\leq \Cint'\|A^{1/2}\nabla \QI (\zeta_{k,j}\phi)\|_{L^2(\Omega\setminus\omega_{k-j})}.$$
Since $\phi\in \Vf$ (i.e., $\QI\phi =0$) and $\phi = I_h \phi$, we have
$$0=\QI\phi=\QI I_h\phi= \overline{\zeta}_{k,j} \QI I_h\phi= \QI I_h(\overline{\zeta}_{k,j}\phi)$$
for all constants $\overline{\zeta}_{k,j}:=|T|^{-1}\int_T \zeta_{k,j}\dx$ with $T\in\tri_H$. Using again (QI3), \eqref{e:cutoff} and \eqref{e:spectralboundh} and recalling that $\QI\phi=0$, this implies
\begin{multline*}
\|A^{1/2}\nabla \QI (\zeta_{k,j}\phi)\|_{L^2(\Omega\setminus\omega_{k-j})}^2\leq\Cint'^2 \|A^{1/2}\nabla \QI I_h((\zeta_{k,j}-\overline{\zeta}_{k,j})\phi)\|_{L^2(\Omega\setminus\omega_{k-j})}^2\\
\lesssim \Cint'^2\Cint^2\|A^{1/2}\nabla( (\zeta_{k,j}-\overline{\zeta}_{k,j})\phi)\|_{L^2(\Omega\setminus\omega_{k-j-1})}\\
\lesssim \Cint'^2\Cint^2\hspace{-3ex}\sum_{T\in\tri_H:T\subset \Omega\setminus\omega_{k-j-1}}\hspace{-3ex}\biggl(\|\zeta_{k,j}-\overline{\zeta}_{k,j}\|_{L^\infty(T)}^2 \|A^{1/2}\nabla\phi\|_{L^2(T)}^2\\+\|\nabla\zeta_{k,j}\|_{L^\infty(T)}^2 \|A^{1/2}(\phi-\QI\phi)\|_{L^2(T)}^2\biggr).
\end{multline*}
The bound for $M_3$ now follows by applying 
Poincar\'e's inequality, the approximation property (QI3) of $\QI$, and the property \eqref{e:cutoff.c} of $\zeta_{k,j}$:
$$M_3\lesssim \Cint^2\Cint' j^{-1}\|A^{1/2}\nabla\phi\|_{L^2(\Omega\setminus\omega_{k-j-2})}.$$

{\it Bound for $M_4$.} 
Similar arguments as before (based on $\QI\phi=0$, the approximation property (QI3) of $\QI$, and the property \eqref{e:cutoff.c} of $\zeta_{k,j}$) lead to the following bound
$$M_4\lesssim \Cint j^{-1}\|A^{1/2}\nabla\phi\|_{L^2(\Omega\setminus\omega_{k-j-1})}.$$

The combination of \eqref{e:decay1} and the bounds for $M_1,\ldots,M_4$ readily yields
\begin{equation*}
\|A^{1/2}\nabla\phi\|_{L^2(\Omega\setminus\omega_{k})}\lesssim \Cint^2\Cint' j^{-1}\|A^{1/2}\nabla\phi\|_{L^2(\Omega\setminus\omega_{k-j-2})}.
\end{equation*}
A sufficiently large enough choice of $j\approx\Cint^2\Cint'$ now establishes the following contraction
\begin{equation}\label{e:decay2}
\|A^{1/2}\nabla\phi\|_{L^2(\Omega\setminus\omega_{k})}\leq \exp(-1)\|A^{1/2}\nabla\phi\|_{L^2(\Omega\setminus\omega_{k-j-2})}.
\end{equation}
We emphasize that the choice of $j$ is independent of $k$ and mesh sizes $H$ and $h$. If the constants $\Cint$, $\Cint'$ in Assumption~\ref{a:QI} are independent of $A$, then $j$ is independent of $A$ as well. 

Repeated application of \eqref{e:decay2} for $k, k\leftarrow (k-j-2), \ldots$ yields
\begin{equation*}
\|A^{1/2}\nabla\phi\|_{L^2(\Omega\setminus\omega_{k})}\leq \exp\left(\left\lfloor\tfrac{k}{j-2}\right\rfloor\right)\|A^{1/2}\nabla\phi\|_{L^2(\Omega)}.
\end{equation*}
This is the assertion up to rephrasing the decay rate in terms of $k,\Cint,\Cint'$ and hiding further uncritical constants in the notation ``$\lesssim$''.
\end{proof}

\subsection{Error estimates for the localized basis}
The error estimate for the localized method from Definition~\ref{definition-full-vmm-truncation} follows from the global error bound of Lemma~\ref{l:ideal} and the decay property of the global correctors established in Lemma~\ref{l:decay} via some algebraic manipulations. 
\begin{theorem}[Energy-error estimate for local coarse approximation]\label{t:main}
If (QI1)--(QI4) are satisfied with constants $\Cint\approx\Cint'\approx 1$ independent of $A$, then there exist $c\approx 1$ such that 
\begin{multline*}
\|A^{1/2}\nabla (u - \uc_k)\|_{L^2(\Omega)}\lesssim \|A^{1/2}\nabla (u - u_h)\|_{L^2(\Omega)}+\alpha^{-1/2}\|H g\|_{L^2(\Omega)}\\
+\sqrt{\ccont}\alpha^{-1/2}H^{-1}e^{-ck}\|g\|_{H^{-1}(\Omega)}.
\end{multline*}
If, moreover, $k\geq \tfrac{1}{2c}\log(\ccont)+\tfrac{2}{c}\log (1/H)$, then 
\begin{equation}\label{e:energyerror}
\|A^{1/2}\nabla (u - \uc_k)\|_{L^2(\Omega)}\lesssim \|A^{1/2}\nabla (u - u_h)\|_{L^2(\Omega)}+\alpha^{-1/2}H\| g\|_{L^2(\Omega)}.
\end{equation}
\end{theorem}
\begin{proof}
The proof of \cite[Theorem~10]{2011arXiv1110.0692M} applies almost verbatim to the present setting. We simply replace the contrast-dependent decay of correctors in \cite{2011arXiv1110.0692M} by our sharper contrast-independent result from Lemma~\ref{l:decay}. Moreover, the proof of \cite{2011arXiv1110.0692M} involves several applications of the norm equivalences \eqref{e:norms} followed by $L^2$ approximation and stability properties. This leads to contrast-dependent constants in \cite{2011arXiv1110.0692M} but can be avoided here by using the approximation and stability in the $A$-weighted $L^2$ norm directly (QI3). However, the  contrast enters our proof via an inverse estimate and leads to the multiplicative constant $\sqrt{\beta/\alpha}$ in front of the exponentially decaying factor $e^{-ck}$. The proper choice of $k$ easily compensates this large constant.
\end{proof}

\begin{remark}[Improved estimates with modified localization]\label{rem:hminus1}
The modified localization of Section~\ref{ss:locmod} allows one to remove the unpleasant constant $H^{-1}$ in front of the exponentially decaying factor in Theorem~\ref{t:main} so that $k\geq \tfrac{1}{2c}\log(\ccont)+\tfrac{1}{c}\log (1/H)$ suffices to establish the error bound \eqref{e:energyerror}. 
\end{remark}

\section{Examples of quasi-interpolation operators}\label{s:qi}
In this section we recall old and introduce new interpolation
operators to be used in the framework presented in Section~\ref{s:abstract}. 

\subsection{$A$-independent quasi-interpolation}\label{ss:classicalQI}
Previous papers, such as \cite{2011arXiv1110.0692M,2012arXiv1211.3551H,2012arXiv1211.5954H}, usually considered a Cl\'ement-type (quasi-)interpolation operator $\QI: V\rightarrow V_H$ presented in \cite{MR1706735}. Given $v\in V$, define a (weighted) Cl\'ement interpolant
\begin{equation}\label{e:clement}
 \QI v := \sum_{z\in\N_H}(\QI v)(z)\lambda_z \quad\text{with nodal values}\;(\QI v)(z):=\frac{(v, \lambda_z)}{(1,\lambda_z)}\quad\text{for}\; z\in\N_H.
\end{equation}
Note that the fine-scale space $\Vf$ can then be characterized as $L^2$-orthogonal complement of $V_H$ in $V_h$. The operator $\QI$ does not depend on the coefficient $A$ and satisfies (local) approximation and stability properties only in unweighted norms \cite{MR1706735}. In particular, there exists a generic  constant $C$ depending only on the shape regularity of the finite element mesh $\triH$ such that for all $v\in V$ and for all $T\in\triH$ it holds
\begin{equation*}
 H_T^{-1}\|v-\QI v\|_{L^{2}(T)}+\|\nabla(v-\QI v)\|_{L^{2}(T)}\leq \Cint \| \nabla v\|_{L^2(\omega_T)}.
\end{equation*}
As shown in \cite{2011arXiv1110.0692M}, this property suffices to establish an optimal a priori error bound for the global version (cf. Definition~\ref{definition-full-vmm-no-truncation}) of the method
\begin{equation}\label{e:estglobal}
\|A^{1/2}\nabla (u - \uc)\|_{L^2(\Omega)}\lesssim \|A^{1/2}\nabla (u - u_h)\|_{L^2(\Omega)}+\alpha^{-1/2}\|H g\|_{L^2(\Omega)}.
\end{equation}
This error estimate does \emph{not} depend on the upper spectral bound $\beta$. Hence, the reliability and accuracy of the global version of the method does not suffer from high contrast. Despite its large computational complexity, the approach may be relevant for upscaling to very coarse meshes, where localization has anyway no effect. 

A further improvement in terms of accuracy (and, hence, the complexity to fall below a given error tolerance) can be achieved by substituting $\QI$ by the modified partition-of-unity-based Cl\'ement interpolation operator $\tilde{\mathcal I}_H$ presented in \cite{MR1736895}. Given $v\in V$, 
$$\tilde{\mathcal I}_H v:=\sum_{z\in\N_H} \frac{(\tilde{\lambda}_z,v)}{(\tilde{\lambda}_z,1)}\lambda_z,\quad\text{where}\quad \tilde\lambda_z(x):=\frac{\lambda_z(x)}{\sum_{z\in\N_H}\lambda_z(x)}.$$ 
Since the $\tilde\lambda_z$, $z\in\N_H$ form a partition of unity up to the boundary, the term $\alpha^{-1/2}\|H g\|_{L^2(\Omega)}$ in \eqref{e:estglobal} can be replaced by data oscillations
$$\biggl(\sum_{z\in\mathcal{N}}\|H (g-g_z)\|_{L^2(\omega_z)}^2\biggr)^{1/2}$$ with some weighted averages $g_z$ of $g$ on the nodal patch $\omega_z$, $z\in\N_H$; we refer to \cite[Section 2]{MR1736895} for details. Further smoothness of the right-hand side $g\in H^1(\Omega)$ then leads to quadratic convergence of the global method to the reference solution independent of contrast.

For both operators localization of the corresponding global basis is possible even for high-contrast coefficients. However, the theory strongly requires (QI3) and (QI4) to be satisfied with constants independent of the contrast and this is not the case in general. Although its performance in the numerical experiments of Section~\ref{s:numexp} is encouraging, the question whether or not the localized version of the method with these classical quasi-interpolation is reliable for high-contrast coefficients remains open.

\subsection{A new quasi-interpolation based on $A$-weighted $L^2$ spaces}\label{ss:qi}

This subsection suggests a new quasi-interpolation operator based on $A$-weighted averages. For this operator we will identify a class of coefficients (with possibly high contrast) that allows us to verify
the conditions (QI1)-(QI4) (see Sections~\ref{ss:characterize}--\ref{ss:verify} below). In particular, this
operator allows for contrast-independent constants $\Cint$ and
$\Cint'$  in (QI3) and (QI4), respectively.

The analysis is technical. To get the main
ideas across, we will only consider the case of scalar coefficients, 
i.e. $A = a I_d$ where $I_d$ is the $d\times d$ identity
matrix and $a \in L^\infty(\Omega)$, $\alpha \le a(x) \le \beta$, for 
almost all $x \in \Omega$. We will further assume that the coefficient
function $a(x)$ is piecewise constant with respect to $\tri_\varepsilon$, 
for some $h \le \varepsilon \le H$, i.e.~we assume that $a(x) = a_\tau$, for all 
$\tau \in \tri_\varepsilon$. Strictly speaking it 
is not necessary that the grids $\tri_\varepsilon$ and
$\tri_H$ are nested but it simplifies the presentation. We assume that 
$\tri_\varepsilon$ is obtained by uniform refinement from  $\tri_H$. Similarly,
$\tri_h$ is obtained by uniform refinement from $\tri_\varepsilon$, and thus
from  $\tri_H$. The extension to isotropic or mildly anisotropic tensor
coefficients and to coefficients that vary mildly (i.e. with benign
contrast but possibly rapidly) within each of the
elements $\tau \in \tri_\varepsilon$ is also straightforward 
(see \cite{PS_IMAJNA2012} for details).

The quasi-interpolation operator is now a coefficient-weighted
generalization of the Cl\'ement-type quasi-interpolation operator presented in Section~\ref{ss:classicalQI} above.
\begin{definition}[A-weighted quasi-interpolation]\label{d:qi}
Given $v \in V_h$, we define
\begin{equation}
\QI v  := \sum_{z \in \mathcal{N}_H} \QI v(z) \lambda_z, \quad \text{with} \quad
\QI v(z) := \frac{\int_{\Omega} a v \lambda_z
  \dx}{\int_{\Omega} a \lambda_z \dx}.
\end{equation}
\end{definition}
Note that the fine-scale space $\Vf$ can then be characterized as orthogonal complement of $V_H$ in $V_h$ with respect to the $A$-weighted $L^2$ scalar product. \subsection{Characterization of feasible high-contrast coefficients}\label{ss:characterize}
To satisfy conditions (QI1)-(QI4) for $\QI$ from Definition~\ref{d:qi} with constants independent of contrast,
we need to make a further assumption on the type of coefficient
distribution. To this end, for each vertex $z \in \mathcal{N}_H$, let $\omega_z=
\text{interior}(\text{supp}(\lambda_z))$, and set $\omega_T :=
\bigcup_{z \in \mathcal{N}_H\cap T} \omega_z$, for all  $T \in \mathcal{T}_H$.

\begin{assumption}
\label{a:quasimono}
We assume that there exists a generic constant $C_{\text{P}}$, independent of the contrast $\beta/\alpha$, such that one of the 
following two Poincar\'e-type inequalities holds for all 
$v \in V_h$ and for all $T \in \mathcal{T}_H$:
\begin{equation}
\label{wPoincare} 
\inf_{c \in \mathbb{R}} \int_{\omega_T} a (v-c)^2 \, \dx \ \lesssim \ C_{\text{P}} H_T^2 
\int_{\omega_T} a |\nabla v|^2 \, \dx,\vspace{-1ex}
\end{equation}
\begin{equation}
\label{wFriedrichs}
\partial \omega_T \cap \partial \Omega \not= \emptyset \quad \text{and} \quad 
\int_{\omega_T} a v^2 \, \dx \ \lesssim \ C_{\text{P}} H_T^2 \int_{\omega_T} a |\nabla v|^2 \, \dx \, .
\end{equation}
\end{assumption}

Since any function $v \in V_h$ is zero on the boundary $\partial
\Omega$, the existence of a constant is guaranteed for any 
strictly positive and uniformly bounded coefficient $a(x)$ by applying
the standard
Poincar\'e/Friedrichs inequality on each of the subregions $\omega_T$.
Whether $C_{\text{P}}$ is independent of the contrast $\beta/\alpha$
depends on the coefficient distribution. 

\subsection{On quasi-monotonicity}
To describe the link between the local
coefficient variation and the 
weighted Poincar\'e inequalities in Assumption
\ref{a:quasimono} in more detail, let 
us consider a generic coarse element $T \in \triH$.

We generalize now the notion 
of quasi-monotonicity coined in \cite{DrSaWi:96} by considering the
following three directed combinatorial graphs $\G_T^{(k)} =
(\N_T,\E_T^{(k)})$, $k=0,1,2$, where $\N_T = \{\tau \in \tri_\varepsilon: \tau
\subset \omega_T\}$ 
and the edges are ordered pairs of vertices. To define the edges we now 
distinguish between three different types of connections.
\begin{definition}
  Suppose that $\gamma^{\tau,\tau'}=\tau \cap \tau'$
  is a non-empty manifold of dimension $k$, for $k=0, 1,2$. The ordered pair
  $(\tau,\tau')$ is an edge in $\E_T^{(k)}$, if and only
  if $a_\tau\lesssim a_{\tau'}$.  The edges in $\E_T^{(k)}$ are said
  to be of \textbf{type-$k$}.
\end{definition}

Quasi-monotonicity is related to the connectivity in these graphs. Let 
$\tau^* = \text{argmax}_{\tau \in \N_T} a_\tau$, i.e. an element in
$\N_T$ where the maximum of $a(x)$ is attained on $\omega_T$. 
\begin{definition}
The coefficient $a$ is type-$k$ quasi-monotone on $\omega_T$, if there is a path in $\G_T^{(k)}$ from any vertex $\tau$ to $\tau^*$.
\end{definition}
Obviously $\E^{(2)} \subset \E^{(1)} \subset \E^{(0)}$, and so type-$k$ 
quasi-monotonicity implies type--$(k-1)$ quasi-monotonicity. The coefficients in 
Figure \ref{fig_coeffs}(a-c) are examples of quasi-monotone coefficients of Type 2, 1 
and 0, respectively. The coefficient in Figure \ref{fig_coeffs}(d) is not quasi-monotone.
\begin{figure}[t]
\begin{center}
\includegraphics[width=0.4\textwidth,angle=270]{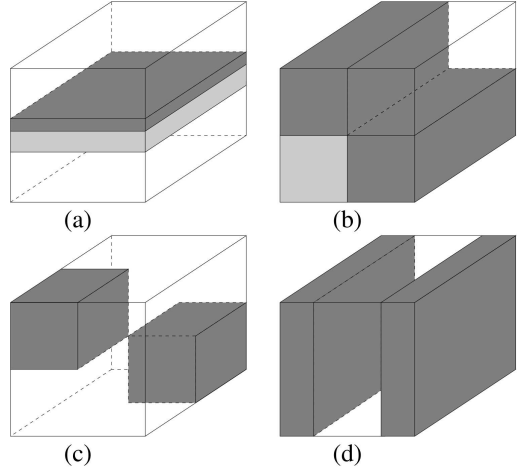}
\caption{\label{fig_coeffs}\small Quasi-monotone coefficient distributions of Type 2, 1 and 0
in (a-c), respectively. A darker color indicates a larger coefficient. A typical non 
quasi-monotone coefficient is shown in (d).}
\end{center}
\end{figure}

The following lemma summarizes the results in
\cite{PS_IMAJNA2012}. It relates the existence of a 
benign constant $C_P$ in Assumption \ref{a:quasimono}, that is independent 
of $\beta/\alpha$, directly to quasi-monotonicity, and the way in which 
$C_P$ depends on the ratio $H/h$ to the type of quasi-monotonicity.
\begin{lemma}
\label{lem_poincare}
If $a$ is type-$k$ quasi-monotone on $\omega_T$, for all $T \in \mathcal{T}_H$ and for some $0 \le k \le d-1$, then Assumption \ref{a:quasimono} holds with
\begin{equation}
C_P \; := \; \left\{ \begin{array}{ll}
1, & \text{if} \ \ k=d-1, \\
1+\log \left({\frac{H}{h}}\right), & \text{if} \ \ k=d-2, \\
{\frac{H}{h}}, & \text{if} \ \ k=0 \ \text{and} \ d=3.
\end{array}\right.
\label{lemma_eq3}
\end{equation}
\end{lemma}

Quasi--monotonicity is a necessary condition. If the coefficient is 
not quasi-monotone, e.g. the situation in Figure \ref{fig_coeffs}(d), 
then there exists a $T \in \mathcal{T}_H$ and $\tau,\ \tau' \in 
\mathcal{N}_T$ with $a_\tau > a_{\tau'}$, such that 
$C_P \ge a_\tau/a_{\tau'}$ (cf. \cite[Prop.~2.11]{PS_IMAJNA2012}).

The coefficient $C_{\text{P}}$ will in general depend on the geometry 
and topology of the coefficient variation. In particular, it depends on the ratio 
$H/\varepsilon$. Restricting ourselves to type--$(d-1)$ quasi-monotone 
coefficients, it is shown in \cite[Section 4]{PS_IMAJNA2012} that 
\begin{equation}
C_P \; \ge \; \left\{ \begin{array}{ll}
1, & \text{if} \ \ d=1, \\
1+\log \left({\frac{H}{\varepsilon}}\right), & \text{if} \ \ d=2, \\
{\frac{H}{\varepsilon}}, & \text{if} \ \ d=3.
\end{array}\right.
\label{lemma_eq3}
\end{equation}
The bounds are sharp and they are attained when $a_\tau \ll a_{\tau^*}$, for all $\tau \in \tri_\varepsilon$ such that $\tau \subset \omega_T$ and $\tau \not= \tau^*$, i.e. when the coefficient is high in only one element $\tau^*\in \tri_\varepsilon$ on $\omega_T$. 

\subsection{Verification of (QI1)-(QI4) for $A$-weighted quasi-interpolation}\label{ss:verify}

To verify conditions (QI1)-(QI4) for the $A$-weighted quasi-interpolation operator $\QI$ in Definition~\ref{a:quasimono}, we need the following two technical lemmas.
For the remainder of this section we assume that $\beta/\alpha \gg H/\varepsilon$, i.e. we consider high-contrast coefficients that do not vary too rapidly relative to the coarse mesh size $H$.
\begin{lemma}[weighted inverse--type estimates]
Let $T \in \tri_H$ and $v_H \in V_H$. Then 
\begin{align}
\label{def:Cstar}
\|v_H\|_{L^\infty(T)} & \le C_{\text{inv},1} \left(\int_{T} a \, \dx\right)^{-1} \int_{T} a |v_H| \, \dx \,,\\
\label{def:Cstarstar}
\|v_H\|_{L^\infty(T)} & \le C_{\text{inv},2} \left(\int_{T} a \, \dx\right)^{-1/2} \|a^{1/2} v_H \|_{L^2(T)}\,,
\end{align} 
with constants $C_{\text{inv},1} \eqsim C_{\text{inv},2} = \mathcal{O}(H/\varepsilon)$ that are independent of the contrast $\beta/\alpha$.
\end{lemma}
\begin{proof}
Let $\|v_H\|_{L^\infty(T)} > 0$; otherwise the results are trivial. Now, set $\widehat{v}_H := \|v_H\|^{-1}_{L^\infty(T)} v_H$. Since $\widehat{v}_H$ is linear on $T$ and equal to $1$ at least at one of the vertices of $T$, it follows as for classical inverse estimates via simple geometric arguments that
\[
\int_\tau |\widehat{v}_H| \, \dx \gtrsim \frac{\varepsilon}{H} |\tau| \quad \text{and} \quad \int_\tau \widehat{v}_H^2 \, \dx \gtrsim \frac{\varepsilon^2}{H^2} |\tau|\,, \quad \text{for all} \ \ \tau \in \tri_\varepsilon\, \ \tau \subset T.
\] 
The implied constants depend only on the dimension $d$ and are independent of the coefficient or of
any geometric parameters. 

Multiplying each of these inequalities by $a_\tau$ and summing over all $\tau \in \tri_\varepsilon$ with $\tau \subset T$ we get
\[
\int_{T} a |\widehat{v}_H| \, d\text{x}  \gtrsim \frac{\varepsilon}{H} \int_{T} a \, \dx \quad \text{and} \quad
\int_{T} a \widehat{v}_H^2 \, d\text{x} \gtrsim \frac{\varepsilon^2}{H^2} \int_{T} a \, \dx\,,
\]
which implies the two inequalities \eqref{def:Cstar} and \eqref{def:Cstarstar}.
\end{proof}

\begin{lemma}
\label{lem:QI4basis}
Let Assumption~\ref{a:quasimono} hold and let $h < \varepsilon$ be sufficiently small. Then, for every $z \in \mathcal{N}_H$, there exists a function $\eta_z \in V_h$ such that $\text{supp}(\eta_z) \subset \omega_z$\,, $\QI \eta_z = \lambda_z 
$ and
\begin{equation}
\label{QI4basis}
\|a^{1/2} \nabla \eta_z\|_{L^2(T)} \le C_{\text{base}}\|a^{1/2} \nabla
\lambda_z\|_{L^2(T)}\,,\ \ \text{for all} \ T \in \tri_H
\end{equation}
with a constant $C_{\text{base}} =\mathcal{O}(H^2/\varepsilon^2)$
that is independent of the contrast $\beta/\alpha$.
\end{lemma}
\begin{proof}
We will only give a complete proof for the case $d=1$. The proof in higher dimensions is very technical and not instructive. We will prove the result by explicitly constructing a suitable piecewise linear function $\eta_z$ that satisfies the required bound. It suffices to work elementwise. 

To simplify the presentation we focus on the particular case where $a|_{\omega^*} \equiv \beta$, for some interval $\omega^* \subset T$ with diameter $\text{diam}(\omega^*) = 2\varepsilon$, and $a(x) = 1$ otherwise (see Figure \ref{fig_basis}). This represents in some sense the worst case scenario. Without loss of generality, we work on the reference element $\widehat{T} = [0,1]$, i.e. $H=1$.
\begin{figure}[t]
\begin{center}
\includegraphics[width=0.3\textwidth]{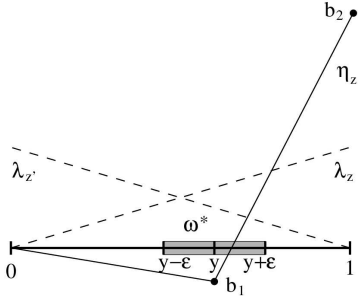}
\caption{\label{fig_basis} Construction of a suitable function $\eta_z$ for Lemma \ref{lem:QI4basis} in one dimension.}
\end{center}
\end{figure}
Let $y$ be the centre of $\omega^*$ which by assumption is a vertex of $\mathcal{T}_h$ and let $\eta_z \in V_h$ be the piecewise linear function with respect to $\{z'=0,y,z=1\}$ that is 0 at $x=0$, $b_1$ at $x=y$ and $b_2$ at $x=1$, as depicted in Figure \ref{fig_basis}. 

Imagining a similar construction in the adjoining element containing $z$, it is easy to see that $\text{supp}(\eta_z) \subset \omega_z$. 
The values of $b_1$ and $b_2$ are chosen such as to ensure that $\QI \eta_z = \lambda_z$. Since
\[
0 = \int_0^1 a \lambda_{z'} \eta_z \dx = \int_0^1 a (1-\lambda_{z}) \eta_z \dx =
\int_0^1 a \eta_z \dx - \int_0^1 a \lambda_{z} \eta_z \dx\,,
\]
this is equivalent to
\begin{equation}
\label{lem43:conditions}
\int_0^1 a \lambda_z = \int_0^1 a \lambda_z \eta_z \dx = \int_0^1 a \eta_z \dx \,.
\end{equation}

An elementary calculation shows that
\begin{align*}
\frac{1}{\beta} \int_0^1 a \lambda_z \dx & = \int_{y-\varepsilon}^{y+\varepsilon} x \dx + \mathcal{O}(\beta^{-1}) = 2y\varepsilon + \mathcal{O}(\beta^{-1}) \\
\frac{1}{\beta} \int_0^1 a \eta_z \dx & = \int_{y-\varepsilon}^{y} \frac{b_1x}{y} \dx 
+ \int_{y}^{y + \varepsilon} \left(b_1\frac{1-x}{1-y} + b_2 \frac{x-y}{1-y} \right) \dx  + \mathcal{O}(\beta^{-1}) \\ & = \frac{\varepsilon}{2} \left[ \frac{4y(1-y) - \varepsilon}{y (1-y)} \, b_1 + \frac{\varepsilon}{1-y} \, b_2 \right] + \mathcal{O}(\beta^{-1})\\
\frac{1}{\beta} \int_0^1 a \lambda_z \eta_z\dx & = \int_{y-\varepsilon}^{y} \frac{b_1x^2}{y} \dx + \int_{y}^{y + \varepsilon} x \left(b_1\frac{1-x}{1-y} + b_2 \frac{x-y}{1-y} \right) \dx  + \mathcal{O}(\beta^{-1}) \\ & = \frac{\varepsilon}{6} \left[ \frac{12y^2(1-y) + (1-2y)2\varepsilon^2 - 3y\varepsilon}{y(1-y)}\, b_1 + \frac{\varepsilon (3y + 2 \varepsilon)}{1-y} \, b_2 \right] + \mathcal{O}(\beta^{-1})
\end{align*}
Due to \eqref{lem43:conditions}, all these expressions need to be equal. We equate the first expression with each of the others and neglect terms of $\mathcal{O}(\beta^{-1})$ (which is justified since we assumed $\beta^{-1} \ll \varepsilon$):
\begin{align}
\label{lem43:first}
\Big(4y(1-y) - \varepsilon\Big) \, b_1 + y \varepsilon \, b_2  &= 4y^2(1-y)\,,\\
\label{lem43:second}
\Big(12y^2(1-y) + (1-2y)2\varepsilon^2 - 3y\varepsilon\Big) \, b_1 + y \varepsilon (3y + 2 \varepsilon) \, b_2 &= 12y^2(1-y)\,.
\end{align}
This uniquely defines $b_1$ and $b_2$ and we see that both values are independent of the contrast~$\beta$. Subtracting $(3y + 2 \varepsilon)$ times  \eqref{lem43:first} from \eqref{lem43:second} and solving for $b_1$ we get
\[
b_1 = - \frac{y^2 (3 - 3y - 2 \varepsilon)}{\varepsilon(2 y - \varepsilon)}
\]
Substituting this back into \eqref{lem43:first} we can get an expression for $b_2$. 

To finish the proof we need to establish \eqref{QI4basis} and show that $C_{\text{base}} =\mathcal{O}(\varepsilon^{-2})$ independent of the contrast $\beta$. Since
\begin{align}
\frac{1}{\varepsilon \beta} \| a^{1/2} \nabla \lambda_z \|^2_{L^2(\widehat{T})} & = \frac{1}{\varepsilon} \int_{y-\varepsilon}^{y + \varepsilon} 1^2 \dx + \mathcal{O}\left(\frac{1}{\beta\varepsilon}\right) = 2 + \mathcal{O}\left(\frac{1}{\beta\varepsilon}\right) \quad \text{and} \nonumber \\
\frac{1}{\varepsilon \beta} \| a^{1/2} \nabla \eta_z \|^2_{L^2(\widehat{T})} & = \frac{1}{\varepsilon} \int_{y-\varepsilon}^{y} \left(\frac{b_1}{y}\right)^2 \dx +\frac{1}{\varepsilon} \int_y^{y+\varepsilon} \left(\frac{b_2-b_1}{1-y}\right)^2 \dx + \mathcal{O}\left(\frac{1}{\beta\varepsilon}\right) \nonumber \\
\label{lem43:bound}
& = \left(\frac{b_1}{y}\right)^2 + \left(\frac{b_2-b_1}{1-y}\right)^2 + \mathcal{O}\left(\frac{1}{\beta\varepsilon}\right)\,,
\end{align}
it suffices to prove that the expression in \eqref{lem43:bound} is $\mathcal{O}(\varepsilon^{-4})$ independent of $\beta$. 

First, it is easy to verify that $-b_1/y$ takes its maximum at $y=\varepsilon$ with a value of $(3-5\varepsilon)/\varepsilon$. Also, it follows from \eqref{lem43:first} that
\[
\frac{b_2-b_1}{1-y} = \frac{4}{\varepsilon} (y-b_1) + \frac{b_1}{y} \le \frac{4}{\varepsilon} (y-b_1) \le \frac{12 y^2 (1 - y)}{\varepsilon^2(2 y - \varepsilon)}
\le \frac{3}{\varepsilon^2}
\]
which completes the proof.
\end{proof}

\begin{proposition}
\label{prop:assumptions}
Under Assumption~\ref{a:quasimono}, the operator $\QI$ from Definition~\ref{a:quasimono} satisfies the conditions (QI1)-(QI4) with constants $\Cint$ and $\Cint'$ 
independent of the contrast $\beta/\alpha$, but depending in general on $H/\varepsilon$.
\end{proposition}

\begin{proof}
(QI1) is satisfied by definition.

To prove (QI2), note that for any $v_H := \sum_{z \in \mathcal{N}_H}
\gamma_z \lambda_z \in V_H$, we have
$$
\QI v_H = \sum_{z \in \mathcal{N}_H} (D^{-1}M\boldsymbol{\gamma})_z
\lambda_z,
$$ 
where $\boldsymbol{\gamma} := (\gamma_z)_{z \in
  \mathcal{N}_H}$, $M$ is the mass matrix with entries $M_{z,z'} :=
\int_{\Omega} a \lambda_z \lambda_{z'} \, \dx$, and $D$ is a diagonal
weighting matrix, with strictly positive entries $D_{z,z} :=
\int_{\Omega} a \lambda_z \, \dx$. Since $M$ is invertible, the mapping $\boldsymbol{\gamma} \to D^{-1}M\boldsymbol{\gamma}$ is bijective, and so the linear map $\QI$ is an isomorphism from $V_H$ to $V_H$.

The proof of (QI3) is analogous to the proof of \cite[Lemma
4.1]{SVZ_SINUM2012}. Let $v_h \in V_h$ and let $T \in \mathcal{T}_H$. Note first that
\[
(\QI v_h(z))^2 \; \le \; \frac{\int_\Omega a v_h^2 \lambda_z \,
  \dx}{\int_\Omega a \lambda_z \, \dx},
\]
which, since $\lambda_z\le 1$, implies
\begin{equation}
\label{ineq:wL2initial}
\int_T a (\QI v_h)^2 \, \dx \; \lesssim \; \sum_{z \in T}
\frac{\int_\Omega a v_h^2 \lambda_z \,
  \dx}{\int_\Omega a \lambda_z \, \dx} \int_T a
\lambda_z^2 \, \dx \; \lesssim \; \int_{\omega_T} a v_h^2 \,
\dx\,
\end{equation}
and consequently 
\begin{equation}
\label{ineq:wL2}
\int_T a (v_h - \QI v_h)^2 \, \dx \; \lesssim \; \int_{\omega_T} a v_h^2 \,
\dx\,.
\end{equation}
Similarly, 
\begin{equation}
\label{ineq:energy}
\int_T a |\nabla \QI v_h|^2 \, \dx \; \lesssim \; \sum_{z \in T}
\frac{\int_\Omega a v_h^2 \lambda_z \,
  \dx}{\int_\Omega a \lambda_z \, \dx} \int_T a
|\nabla \lambda_z|^2 \, \dx \,.
\end{equation}
Since $|\nabla \lambda_z|^2 = c_{T,z} H^{-2}_T$ on $T$, for some constant $c_{T,z}$ that only
depends on the shape of $T$, it can be pulled out of the last integral in \eqref{ineq:energy}. Also, 
using the inverse estimate \eqref{def:Cstar} with $v_H = \lambda_z$ we have
\[
\int_{T} a \, \dx\; \le \; C_{\text{inv},1} \int_{T} a \lambda_z \, \dx\,.
\]
Combined with \eqref{ineq:energy} and using again that $\lambda_z \le 1$, this leads
to
\begin{equation}
\label{ineq:energy2}
\int_T a |\nabla(v_h - \QI v_h)|^2 \, \dx \; \lesssim \;
C_{\text{inv},1} H^{-2}_T \int_{\omega_T} a v_h^2 \,
\dx\,.
\end{equation}

The result now follows (as in the proof of \cite[Lemma 4.1]{SVZ_SINUM2012}) from \eqref{ineq:wL2}, \eqref{ineq:energy2} 
and Assumption \ref{a:quasimono} by summation over all $T\in \mathcal{T}_H$, since either $\partial \omega_T
\cap \partial \Omega \not= \emptyset$ or $\{\lambda_z\}$ forms a partition of unity on $\omega_T$
and thus $\QI$ preserves constants. The constant in (QI3) satisfies $C_{\text{qip}}\lesssim \sqrt{C_P\,C_{\text{inv},1}}$. In the worst case, for $d=3$, we have $C_{\text{qip}} = \mathcal{O}(H/\varepsilon)$. (In fact, as we can see above, the factor $\sqrt{C_{\text{inv},1}}$ only appears in the bound of the energy error not in the $L^2$ part in (QI3). The $L^2$ part in (QI3) can be bounded with a constant independent of $H/\varepsilon$ in one dimension and the constant only grows logarithmically with $H/\varepsilon$ in two dimensions.)

For (QI4), we proceed as in the proof of 
\cite[Lemma 1]{2011arXiv1110.0692M}.  In Lemma \ref{lem:QI4basis}, we already proved (QI4) for any nodal basis function 
$\lambda_z$, $z \in \mathcal{N}_H$.
Now, to prove (QI4) for an arbitrary $v_H := \sum_{z \in \mathcal{N}_H} v_H(z)
\lambda_z \in V_H$, we choose $v_h = v_H + \sum_{z \in
  \mathcal{N}_H} (v_H(z) - \QI v_H(z)) \eta_z \in V_h$\,, where $\eta_z \in V_h$ is as defined in Lemma \ref{lem:QI4basis}. 
The facts that $\QI v_h = v_H$ and $\text{supp} \, v_h \subset \text{supp}
\, v_H$ follow immediately. 

To establish the stability bound, we use the inverse estimate \eqref{def:Cstarstar} together with \eqref{QI4basis}, \eqref{ineq:wL2} and the fact that $|\nabla \lambda_z|^2 = c_{T,z} H_T^{-2}$ on $T$, and we get
\begin{align*}
\| a^{1/2} \nabla v_h\|^2_{L^2(T)} \ & \lesssim \ \| a^{1/2} \nabla v_H\|^2_{L^2(T)} + \sum_{z \in T} |v_H(z) - \QI v_H(z)|^2 \, \| a^{1/2} \nabla \eta_z\|^2_{L^2(T)} \\
& \le \ \| a^{1/2} \nabla v_H\|^2_{L^2(T)} + \sum_{z \in T} |v_H(z) - \QI v_H(z)|^2 \, C_{\text{base}} \left(\int_T  a(x)\right) c_{T,z} H_T^{-2} \\
& \lesssim \ \| a^{1/2} \nabla v_H\|^2_{L^2(T)} + C_{\text{inv},2}^2C_{\text{base}} \, H_T^{-2}\,\| a^{1/2}(v_H - \QI v_H) \|^2_{L^2(T)} \\ 
& \lesssim \ C_{\text{inv},2}^2C_{\text{base}} C_P \, \| a^{1/2} \nabla v_H\|^2_{L^2(\omega_T)}\,.
\end{align*}
Summing over all $T \in \mathcal{T}_H$, we then obtain (QI4) with a constant $C_{\text{qip}}' \lesssim C_{\text{inv},2}\sqrt{C_{\text{base}} C_P}$. In the worst case, for $d=3$, we may have $C_{\text{qip}}' = \mathcal{O}\left((H/\varepsilon)^{5/2}\right)$.
\end{proof}

\subsection{Alternative quasi-interpolation operators}
\label{ss:alternative}

The previous interpolation operators are associated with $L^2$ and $A$-weighted $L^2$ projections onto classical finite elements. While those projections are global operators, we will now consider local projections. In other works \cite{Brown.Peterseim:2014,Peterseim2014,Gallistl.Peterseim:2015}, local projections turned out to be superior over the (weighted) $L^2$-projections and their corresponding non-projective quasi-interpolations. 
\begin{definition}[$A$-weighted projective quasi-interpolation]\label{d:pqi}
Given $v \in V_h$, we define
\begin{equation}\label{e:defproj}
\QI^{\operatorname{proj},A} v  := \sum_{z \in \mathcal{N}_H} \mathcal{P}_z v(z) \lambda_z,
\end{equation}
where, for any $z\in\mathcal{N}_H$, $\mathcal{P}_z v\in V_H\vert_{\omega_z}$ is the local weighted $L^2$ projection onto the coarse finite element space restricted to the nodal patch $\omega_z$, i.e.,
\begin{equation}\label{e:defproj1}
\int_{\omega_z}a\, \mathcal{P}_z v\, w_H\dx  = \int_{\omega_z}a\,v\, w_H\dx\qquad\text{for all }w_H\in V_H\vert_{\omega_z}.
\end{equation}
\end{definition}

Since $\QI^{\operatorname{proj},A}$ is a projection,  (QI2) and (QI4) are satisfied trivially here with $C_{\text{qip}}'=1$. Assumption (QI1) is again satisfied by definition. Assumption (QI3) can be verified as in the proof of Proposition \ref{prop:assumptions} for $\QI$ with a constant $C_{\text{qip}}$ that is independent of $\beta/\alpha$ but does depend again on $H/\varepsilon$. The key observation is that the local mass matrix $M_z$ with entries $M_{z;\zeta,\zeta'} := \int_{\omega_z} a \,\lambda_\zeta \,\lambda_{\zeta'} \dx$ associated with the patch $\omega_z$ is spectrally equivalent to $D_z := \text{diag}(M_z)$. Let  $N_z := \text{dim}\left(V_H\vert_{\omega_z}\right)$, then this means that
\begin{equation}
\label{ineq:spectral_local}
\mu_{\min,z} \textbf{w}^T  D_z\textbf{w} \; \le \; \textbf{w}^T  M_z \textbf{w} \; \le \; \mu_{\max,z} \textbf{w}^T  D_z \textbf{w}, \quad \text{for all} \  \textbf{w} \in \mathbb{R}^{N_z} \,,
\end{equation}
which in turn guarantees that 
\[
\left(\mathcal{P}_zv(z)\right)^2 \int_{\omega_z}a\,\lambda_z^2 \dx \le \mu^{-1}_{\min,z} \int_{\omega_z}a\, \left(\mathcal{P}_z v\right)^2 \dx \le \mu^{-1}_{\min,z} \int_{\omega_z}a\, v^2 \dx
\]
and allows to establish a bound akin to \eqref{ineq:wL2initial}. The remainder follows as in the proof of Proposition \ref{prop:assumptions}.

Crucially, 
we require that $\mu_{\min,z}^{-1}$ in \eqref{ineq:spectral_local} can be bounded independently of $\beta/\alpha$. Note that $\mu_{\min,z}$ is also the smallest eigenvalue of $D^{-1}_z M_z$. As in Lemma \ref{lem:QI4basis}, we show this only for the special case of $d=1$ and $a|_{\omega^*} \equiv \beta$, in some interval $\omega^* \subset \omega_z$ with diameter $\text{diam}(\omega^*) = 2\varepsilon$, and $a(x) = 1$ otherwise, as depicted in Figure \ref{fig_basis}. Without loss of generality, we assume that $\omega^* \subset T$ for one of the two elements $T$ making up $\omega_z$ and that $H=1$ again. An elementary calculation shows that 
\[
M_z = 2\beta \varepsilon \, \left[ \begin{array}{ccc}
(1-y)^2 + \varepsilon^2/3 + \mathcal{O}\left(\frac{1}{\beta \varepsilon}\right) & y(1-y) - \varepsilon^2/3  + \mathcal{O}\left(\frac{1}{\beta \varepsilon}\right) & 0\\
y(1-y) - \varepsilon^2/3  + \mathcal{O}\left(\frac{1}{\beta \varepsilon}\right) & y^2 + \varepsilon^2/3 + \mathcal{O}\left(\frac{1}{\beta \varepsilon}\right) & \frac{1}{12\beta \varepsilon}\\
0 & \frac{1}{12\beta \varepsilon} & \frac{1}{6\beta \varepsilon}
\end{array}\right].
\]
Considering first $y \gg \varepsilon$ and $1-y \gg \varepsilon$ and ignoring terms of $\mathcal{O}\left(\frac{1}{\beta \varepsilon}\right)$ and  $\mathcal{O}\left(\frac{\varepsilon^4}{y^4 (1-y)^4}\right) $, we get
\[
D_z^{-1} M_z = \left[ \begin{array}{ccc}
1 & \frac{y}{1-y} - \frac{\varepsilon^2}{3(1-y)^3} & 0\\
\frac{1-y}{y} - \frac{\varepsilon^2}{3y^3} & 1 & 0 \\
0 & \frac{1}{2} & 1
\end{array}\right]\,.
\]
The eigenvalues of $D_z^{-1} M_z$ satisfy
\[
0 = \text{det}(\mu I - D_z^{-1} M_z) = (\mu - 1) \left(\mu^2 -2 \mu +  \frac{\varepsilon^2}{3y^2(1-y)^2} \right)
\]
leading to $\mu_{\min,z} = \frac{\varepsilon^2}{6y^2(1-y)^2}$ which is independent of $\beta$. For $y  =  \mathcal{O}(\varepsilon)$ or $1-y = \mathcal{O}(\varepsilon)$,
it is even possible to bound $\mu^{-1}_{\min,z}$ independently of $\varepsilon$.

In the numerical experiments of Section~\ref{s:numexp} we will also consider the non-weighted variant $\QI^{\operatorname{proj}}$ that is defined in the same way with classical $L^2$ inner products ($a=1$) in \eqref{e:defproj1}.

\section{Numerical experiments}\label{s:numexp}
Three numerical experiments shall illustrate our theoretical results and illuminate their sharpness and limitations. Numerical experiments with highly oscillatory and high-contrast coefficients have already been documented in \cite{2011arXiv1110.0692M,2012arXiv1211.5954H,2012arXiv1211.3551H,HM_arXiv2013,2013arXiv1312.5922P}. While those results were based on the classical coefficient-independent interpolation defined in Section~\ref{ss:classicalQI}, this section considers several choices of interpolation operators and investigates the possible benefit of using $A$-weighted interpolation operators of Sections~\ref{ss:qi} and \ref{ss:alternative} when high contrast is present. 

\subsection{High-contrast blocks}\label{ss:numexp1}
The first model problem considers a two-phase coefficient with simple topology. 
The precise data of the first model problem is as follows,
\begin{align}\label{e:numexp1}
\Omega &:= ]0,1[^2;\\
g(x) &:= \begin{cases}
     0,& x\in[0,\tfrac{1}{2}[\times [0,1],\\
     1,& x\in[\tfrac{1}{2},1]\times [0,1];
    \end{cases},\\
A(x) &:= \begin{cases}
        \beta, & x\in [\tfrac{11}{32},\tfrac{5}{32}]\times[\tfrac{8}{32},\tfrac{11}{32}]\cup[\tfrac{5}{32},\tfrac{11}{32}]\times[\tfrac{8}{32},\tfrac{19}{32}],\\
        1,&\text{elsewhere.}
       \end{cases}
\end{align}
Since the lower bound of $A$ is one, the parameter $\beta\geq 1$ reflects the contrast. We consider the following values for the contrast, $\beta=1,10,\ldots,10^6$. The numerical experiment aims to study the dependence between these choices of the parameter and the accuracy of the numerical methods.

\begin{figure}
\begin{center}
\includegraphics[width=0.2\textwidth]{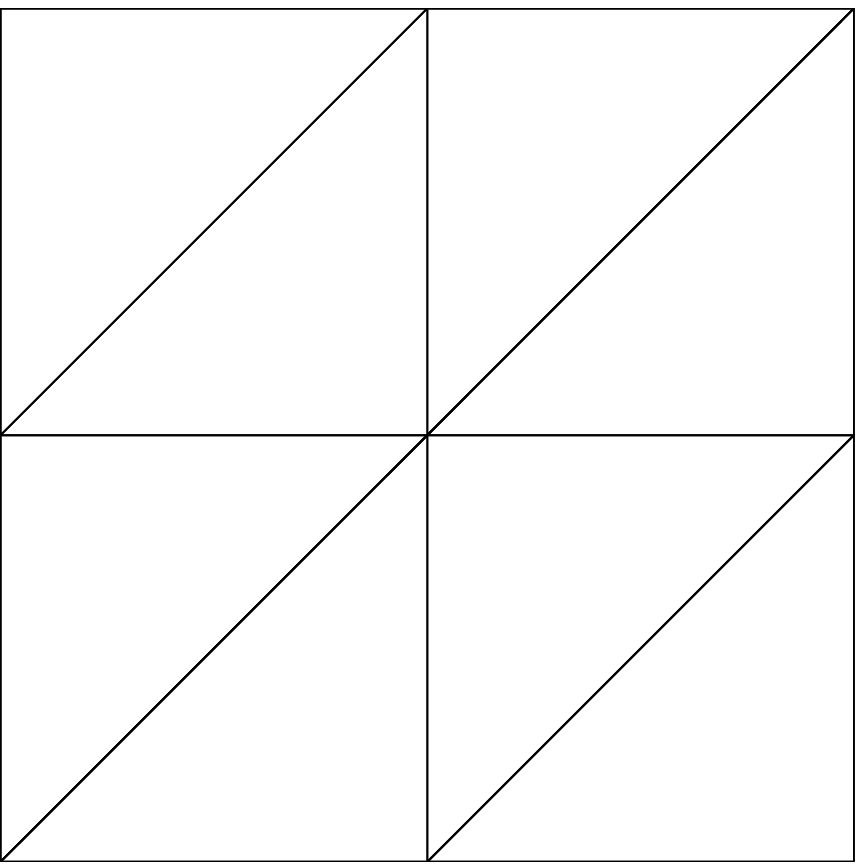}\hspace{2ex}\includegraphics[width=0.2\textwidth]{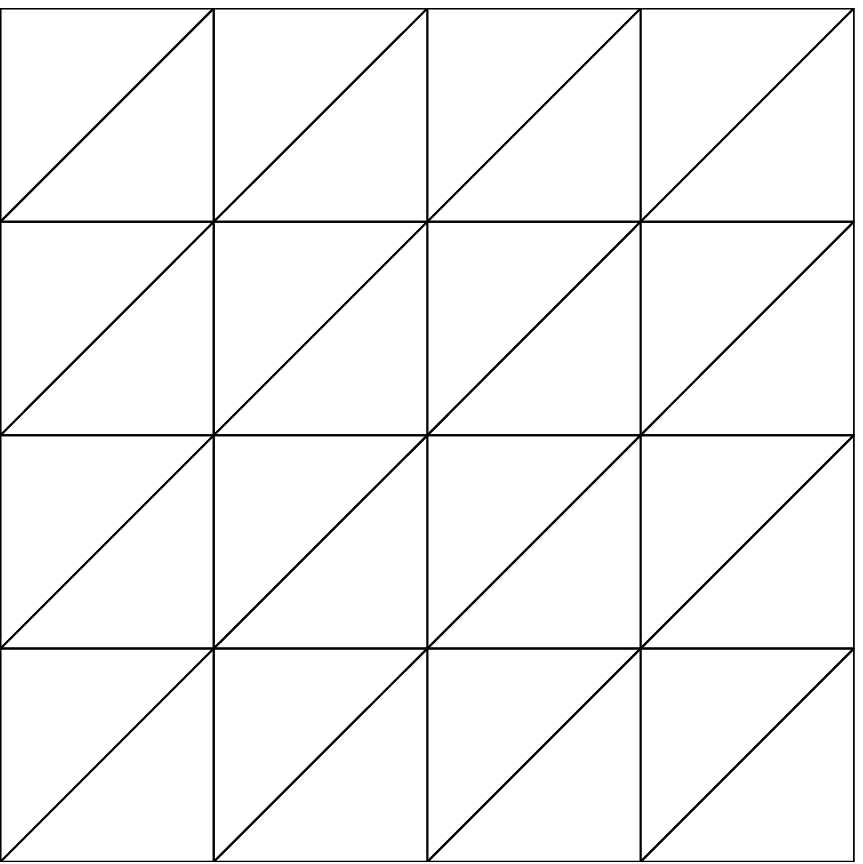}
\caption{Uniform triangulations of the unit square used as coarse meshes in the numerical experiments of Section~\ref{s:numexp}.\label{fig:meshes}}
\end{center}
\end{figure}
Consider the uniform coarse meshes with mesh widths $\sqrt{2}H=2^{-1},\ldots,2^{-6}$ of $\Omega$ as depicted in Figure~\ref{fig:meshes}. The reference mesh $\tri_h$ is derived by uniform mesh refinement of the coarse meshes and has maximal mesh width $h=2^{-8}/\sqrt{2}$. The corresponding $P1$ conforming finite element approximation on the reference mesh $\tri_h$ is denoted by $V_h$. We consider the reference solution $u_h\in V_h$ of \eqref{e:modelweak} with data given in \eqref{e:numexp1} and compare it with coarse scale approximations $\uc_k\in \Vc_k$ (cf. Definition~\ref{definition-full-vmm-truncation}) depending on the coarse mesh size $H$, the localization parameter $k$ and the underlying quasi-interpolation operator $\QI$. We consider four different quasi-interpolation operators, the $A$-independent variant $\QI$ defined in Section~\ref{ss:classicalQI}, the $A$-weighted version $\QI^A$ from Definition~\ref{d:qi}, the $A$-independent operator $\QI^{\operatorname{proj}}$ with projection property defined in Section~\ref{ss:alternative} and its $A$-weighted version $\QI^{\operatorname{proj},A}$.
\begin{figure}[tb]
\begin{center}
\subfigure[\label{fig:numexp1H_msfem}Results for $\QI$.
]{\includegraphics[width=0.38\textwidth]{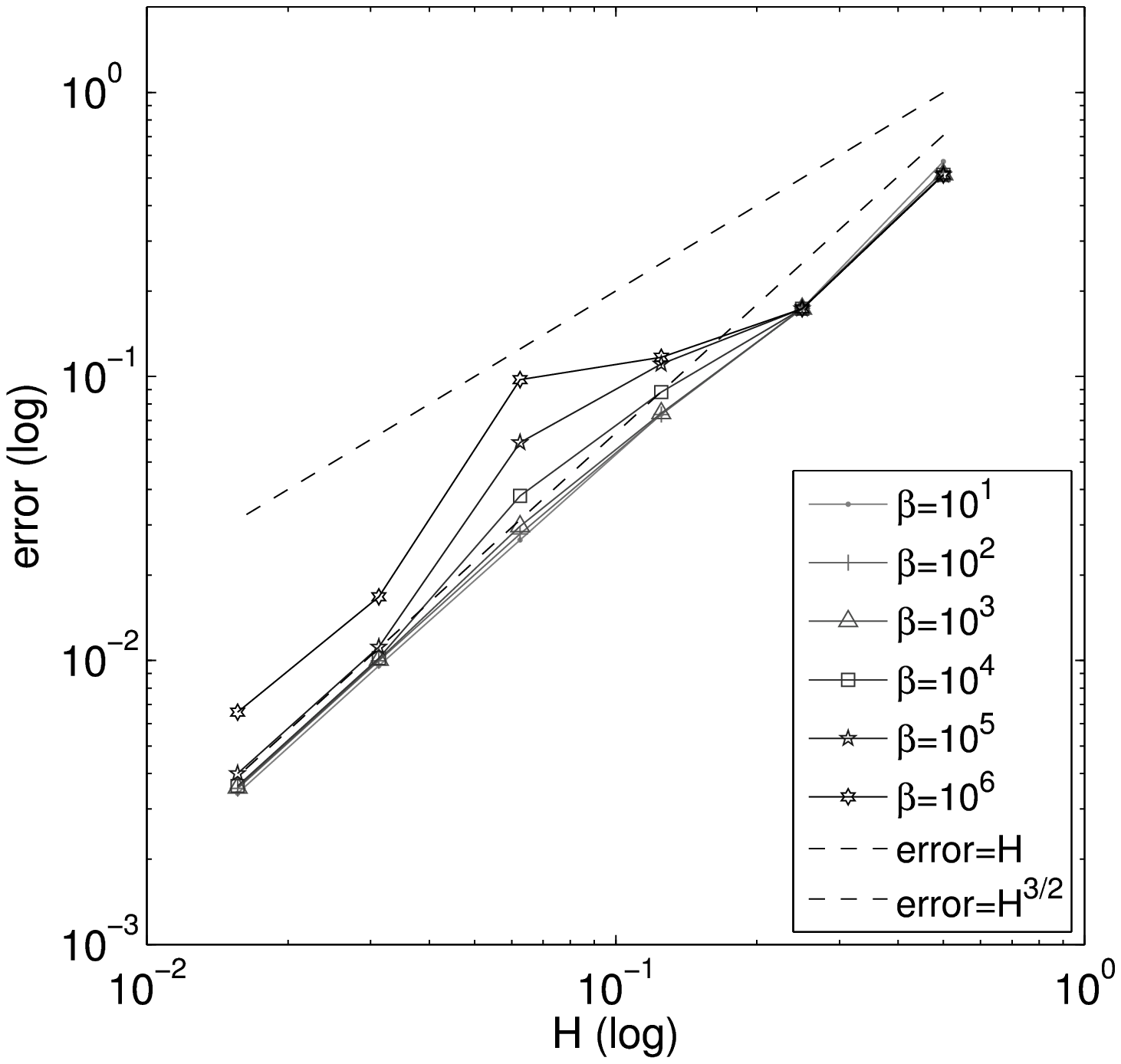}}
\subfigure[\label{fig:numexp1H_msfema}Results for $\QI^A$.
]{\includegraphics[width=0.38\textwidth]{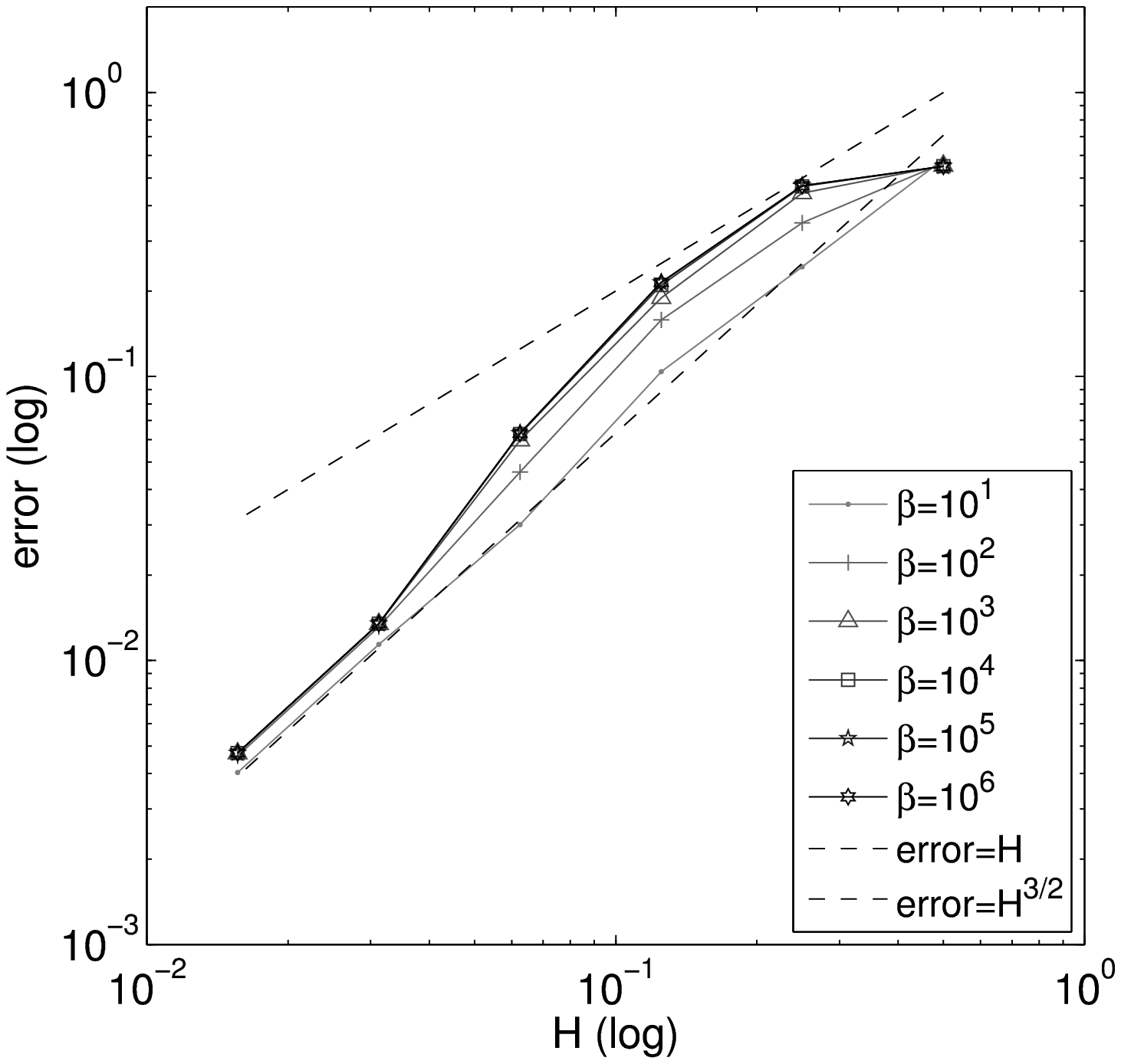}}\\
\subfigure[\label{fig:numexp1H_msfem1}Results for $\QI^{\operatorname{proj}}$.
]{\includegraphics[width=0.38\textwidth]{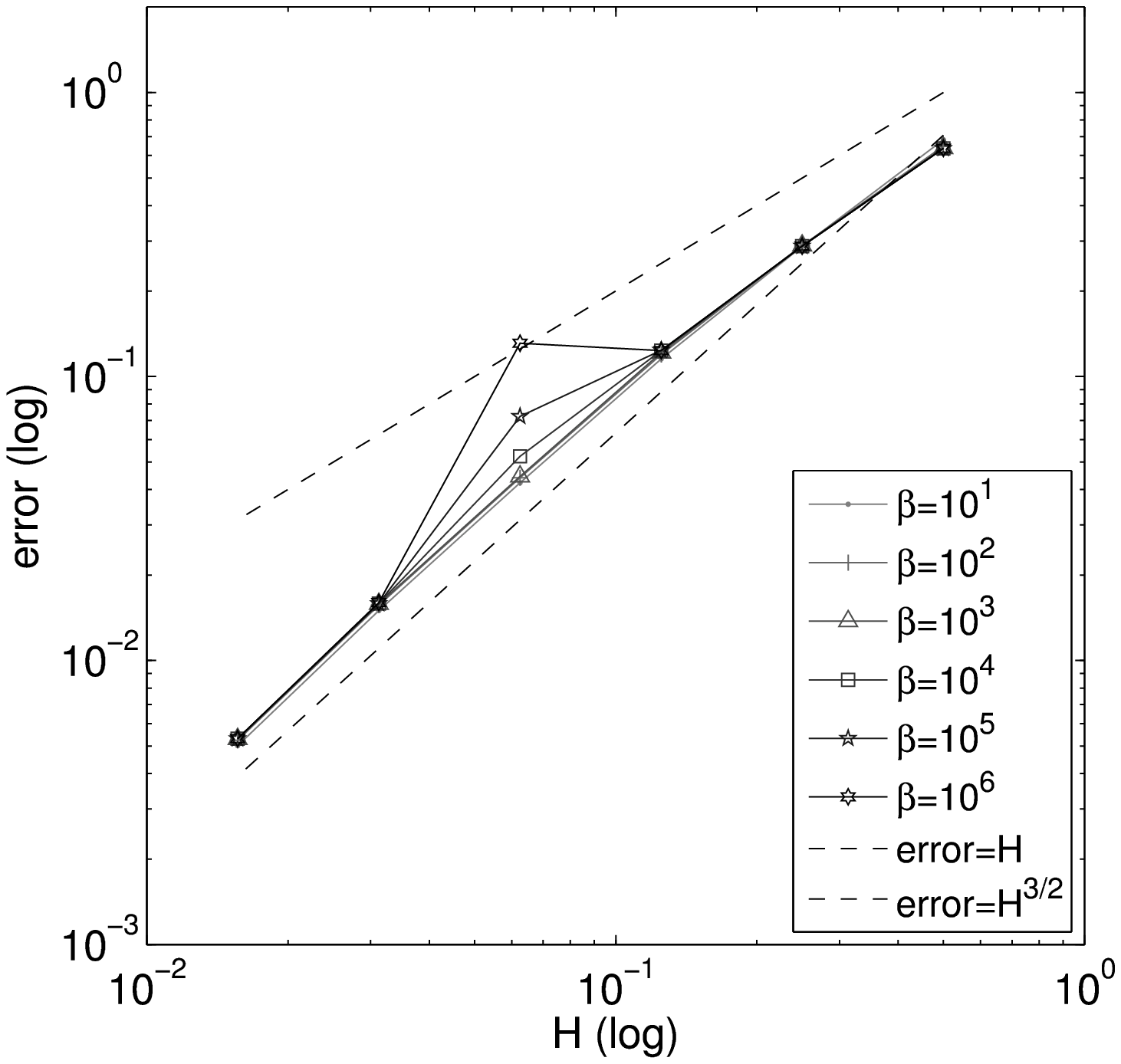}}
\subfigure[\label{fig:numexp1H_msfem1a}Results for $\QI^{\operatorname{proj},A}$.
]{\includegraphics[width=0.38\textwidth]{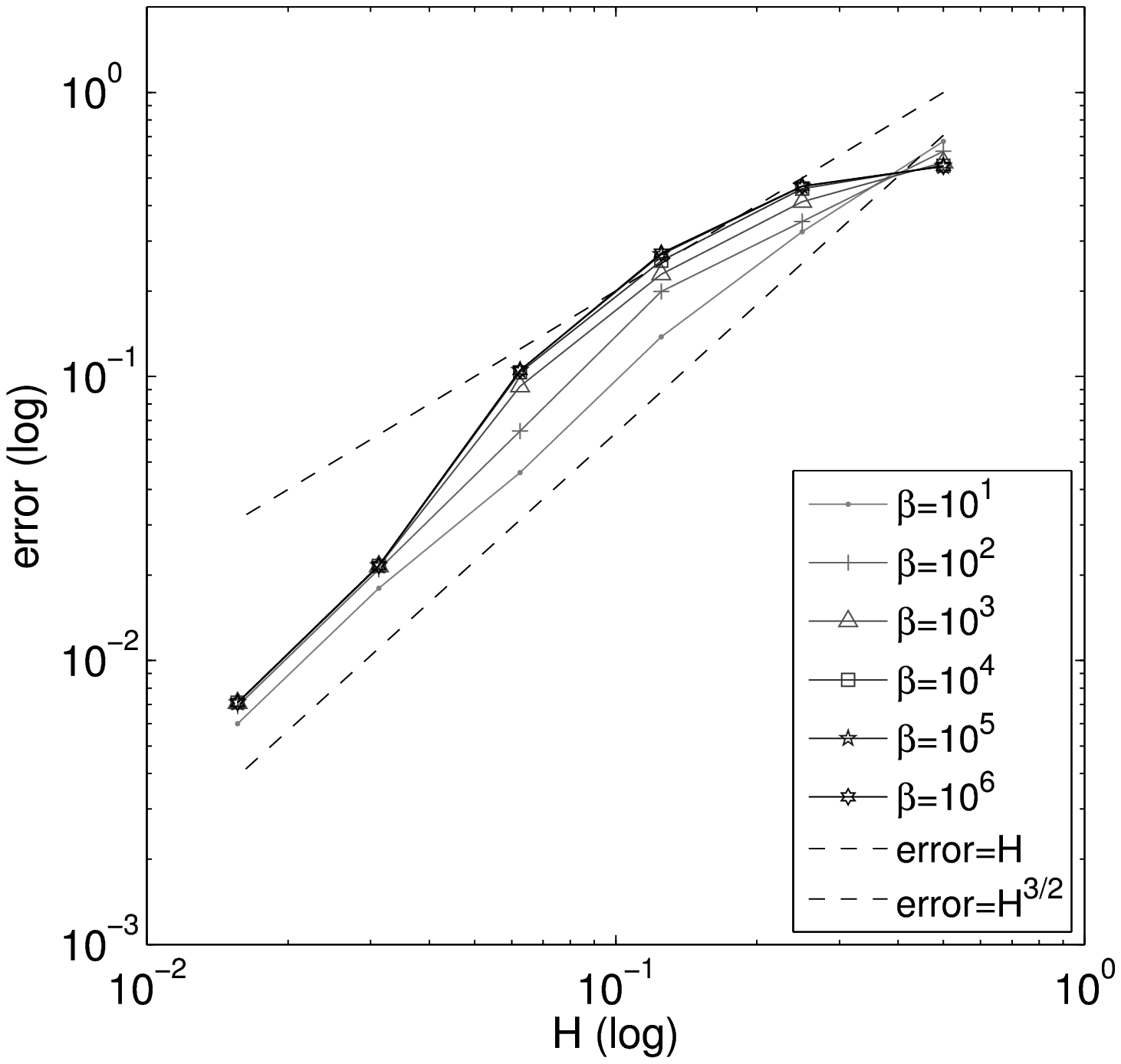}}\\
\end{center}
\caption{Numerical experiment of Section~\ref{ss:numexp1}: Results for high-contrast blocks with several choices of the contrast parameter $\beta$ ($\alpha=1$) depending on the coarse mesh size $H$. The reference mesh size $h=2^{-8}$ remains fixed. The localization parameter is tied to the coarse mesh size via the relation $k=|\log_2 H|+1$. \label{fig:numexp1H}}
\end{figure}

The results are visualized in Figures~\ref{fig:numexp1H} and~\ref{fig:numexp1k}. Figure~\ref{fig:numexp1H} shows the relative energy errors $\|A^{1/2}\nabla (u_h-\uc_{k(H)})\|/\|A^{1/2}\nabla u_h\|$ depending on the coarse mesh size $H$ for several choices of the contrast parameter $\beta=1,10,\ldots,10^6$. The localization parameter $k$ is tied to $H$ via the relation $k=k(H)=|\log_2 H|+1$ (without any dependence on $\beta$). 
For all choices of interpolation operators, only a very mild dependence on $\beta$ can be observed. In particular, all errors are below the reference curve $H$. Asymptotically, the experimental convergence rate $H^{3/2}$ is observed. This high rate is related to certain $L^2$ or $L^2(A)$ orthogonality properties of the interpolation operators as indicated in Section~\ref{ss:classicalQI}.

Figure~\ref{fig:numexp1k} aims to illustrate the role of the localization parameter. It depicts relative energy errors $\|A^{1/2}\nabla (u_h-\uc_{k(H)})\|/\|A^{1/2}\nabla u_h\|$ depending on the coarse mesh size $H$ for fixed contrast $\beta=10^6$ and several choices of the localization parameter $k=1,2,3,\ldots,8$. (We also show relative errors of the standard conforming $P1$-FEM on the coarse meshes for comparison.) We observe a much faster decay of the error when $k$ is increased for the methods that are based on $A$-weighted interpolation. For these methods, a fixed choice of $k=2$ or $k=3$ already gives very good accuracy for the range of coarse meshes considered. For these small choices of $k$, the methods based on $A$-independent interpolation are strongly affected by the high contrast. They are more accurate only for sufficiently large $k$.
\begin{figure}[tb]
\begin{center}
\subfigure[\label{fig:numexp1k_msfem}Results for $\QI$.
]{\includegraphics[width=0.38\textwidth]{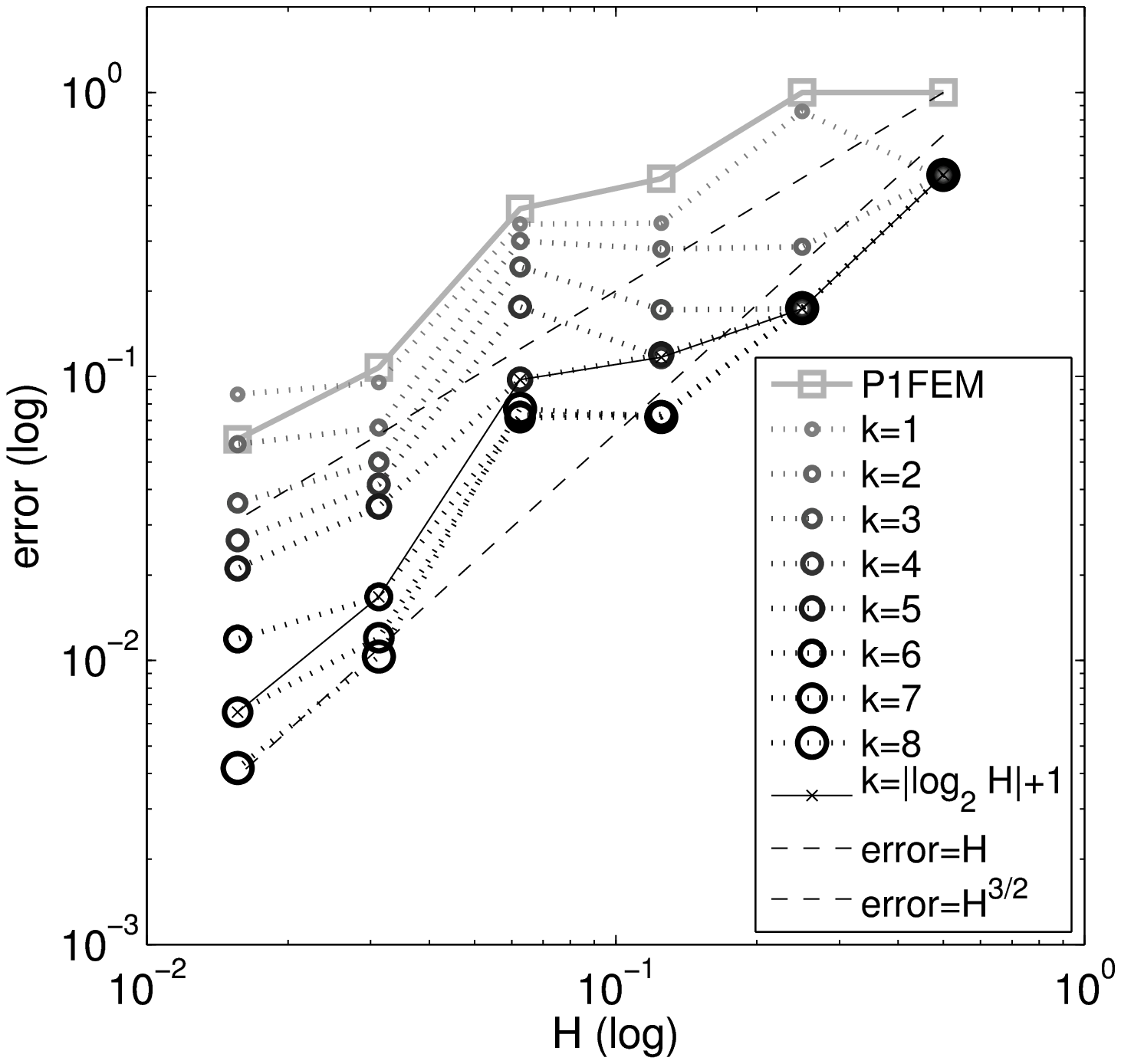}}
\subfigure[\label{fig:numexp1k_msfema}Results for $\QI^A$.
]{\includegraphics[width=0.38\textwidth]{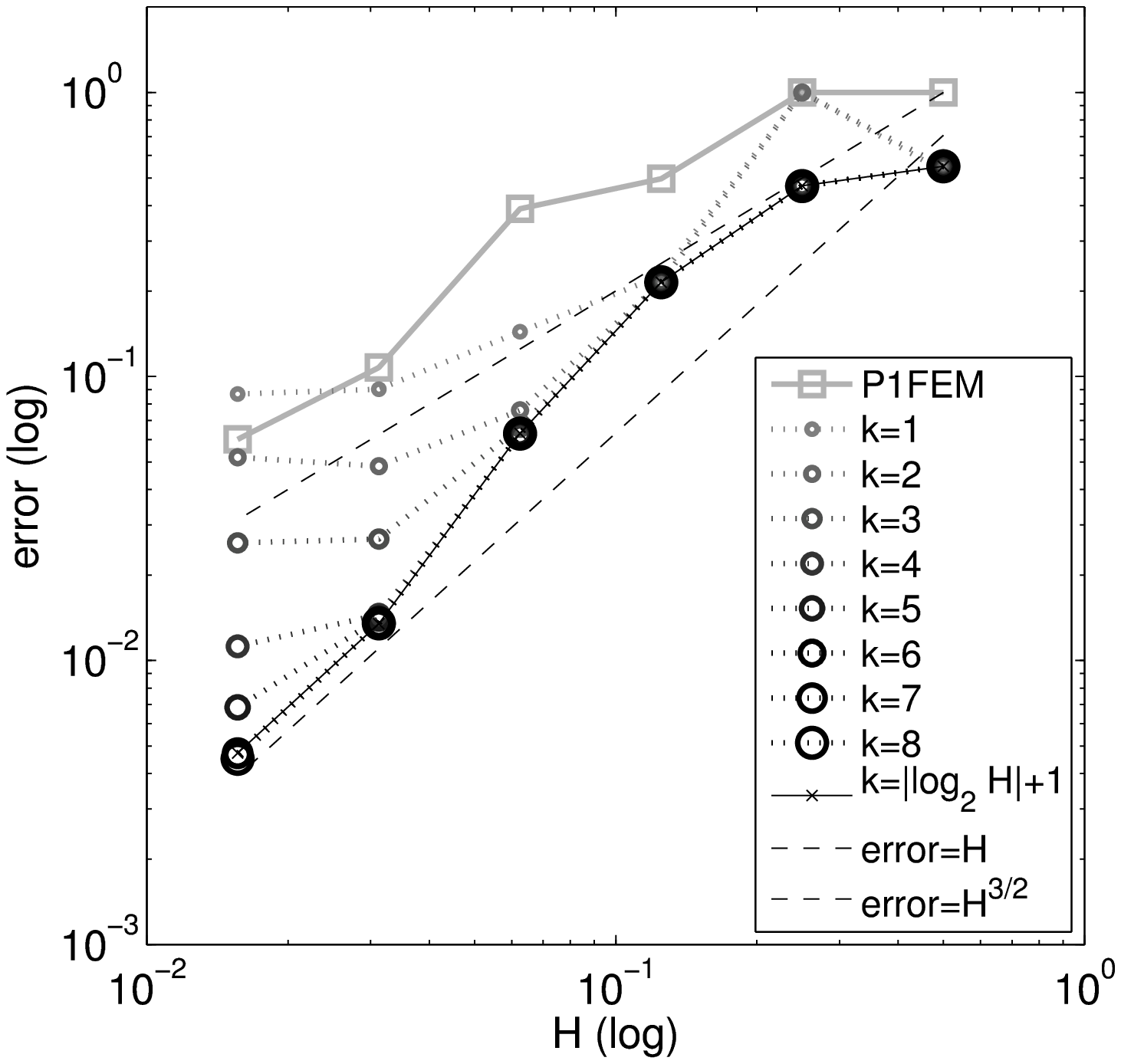}}\\
\subfigure[\label{fig:numexp1k_msfem1}Results for $\QI^{\operatorname{proj}}$.
]{\includegraphics[width=0.38\textwidth]{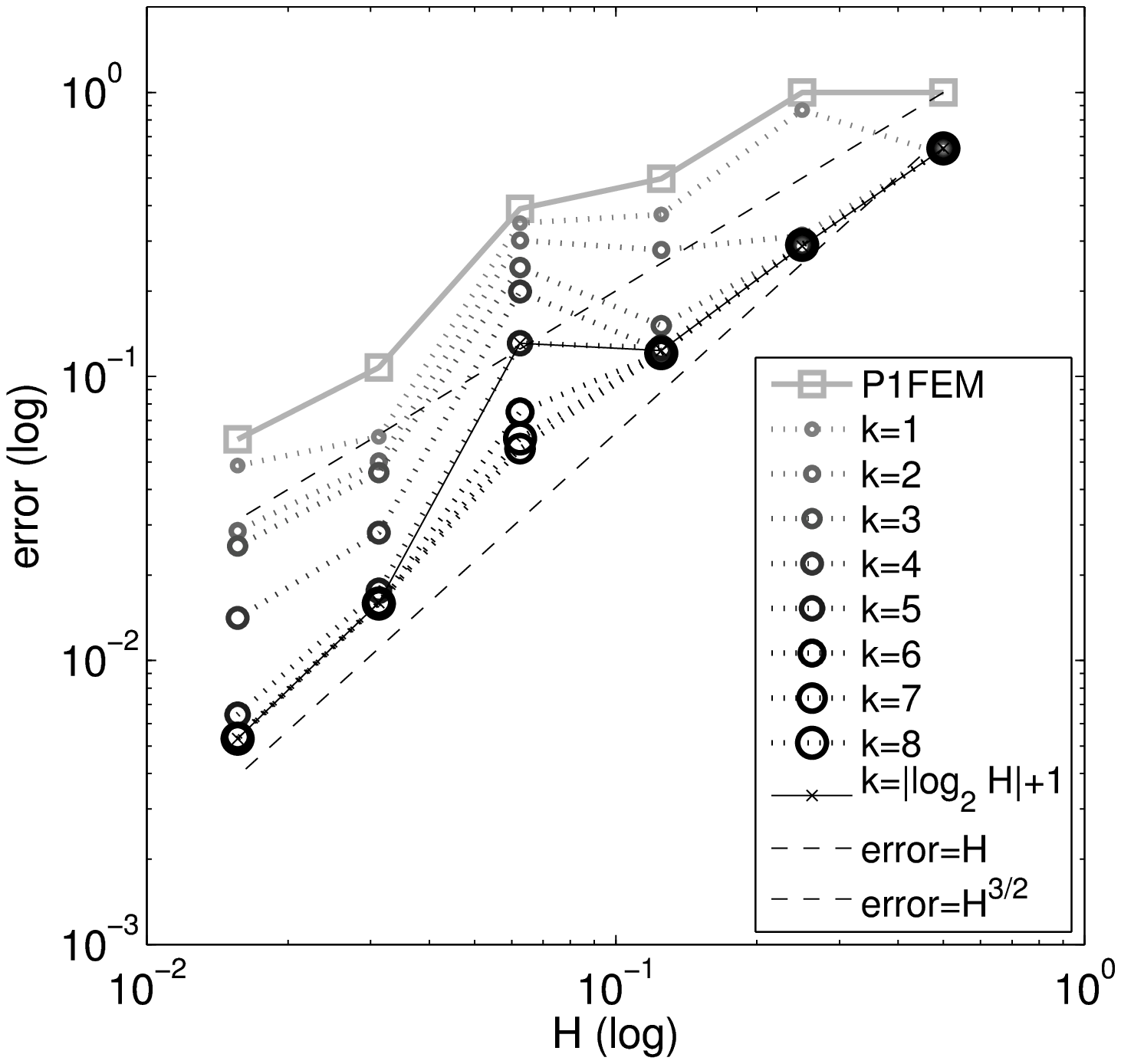}}
\subfigure[\label{fig:numexp1k_msfem1a}Results for $\QI^{\operatorname{proj},A}$.
]{\includegraphics[width=0.38\textwidth]{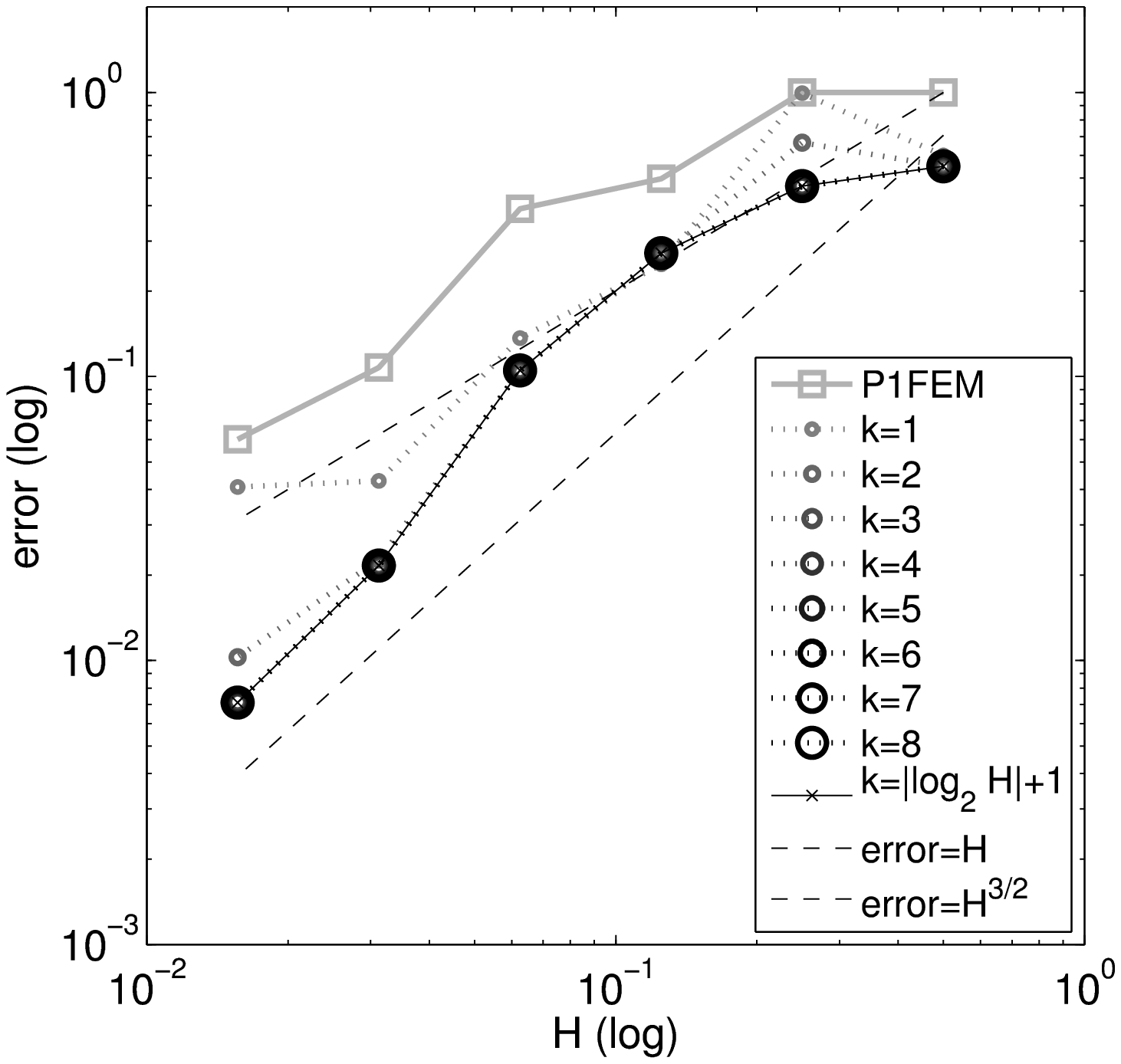}}\\
\end{center}
\caption{Numerical experiment of Section~\ref{ss:numexp1}: Results for high-contrast blocks with contrast parameter $\beta=10^6$ depending on the coarse mesh size $H$. The reference mesh size $h=2^{-8}$ remains fixed. The localization parameter $k$ is varied between $1$ and $8$. \label{fig:numexp1k}}
\end{figure}

\subsection{High-contrast channels}\label{ss:numexp2}
The second experiment repeats the previous computations for a different two-phase coefficient. The precise data of the second model problem is as follows,
\begin{align}\label{e:numexp2}
\Omega &:= ]0,1[^2;\\
g(x) &:= \begin{cases}
     0,& x\in[0,\tfrac{1}{2}[\times [0,1],\\
     1,& x\in[\tfrac{1}{2},1]\times [0,1];
\end{cases}\\
A(x)=A(x_1,x_2)&:=A_1(x_1,x_2)+A_1(x_2,x_1),\text{ where}\\
A_1(x) &:= \begin{cases}
        \beta/2, & x\in [\tfrac{8}{32},\tfrac{9}{32}]\times[\tfrac{1}{32},\tfrac{31}{32}]\cup[\tfrac{10}{32},\tfrac{11}{32}]\times[\tfrac{1}{32},\tfrac{31}{32}],\\
        1,&\text{elsewhere.}
       \end{cases}
\end{align}
Again, the parameter $\beta\geq 1$ reflects the contrast and the numerical experiment aims to study the dependence between this parameter and the accuracy of the numerical methods. 
\begin{figure}[tb]
\begin{center}
\subfigure[\label{fig:numexp2H_msfem}Results for $\QI$.
]{\includegraphics[width=0.38\textwidth]{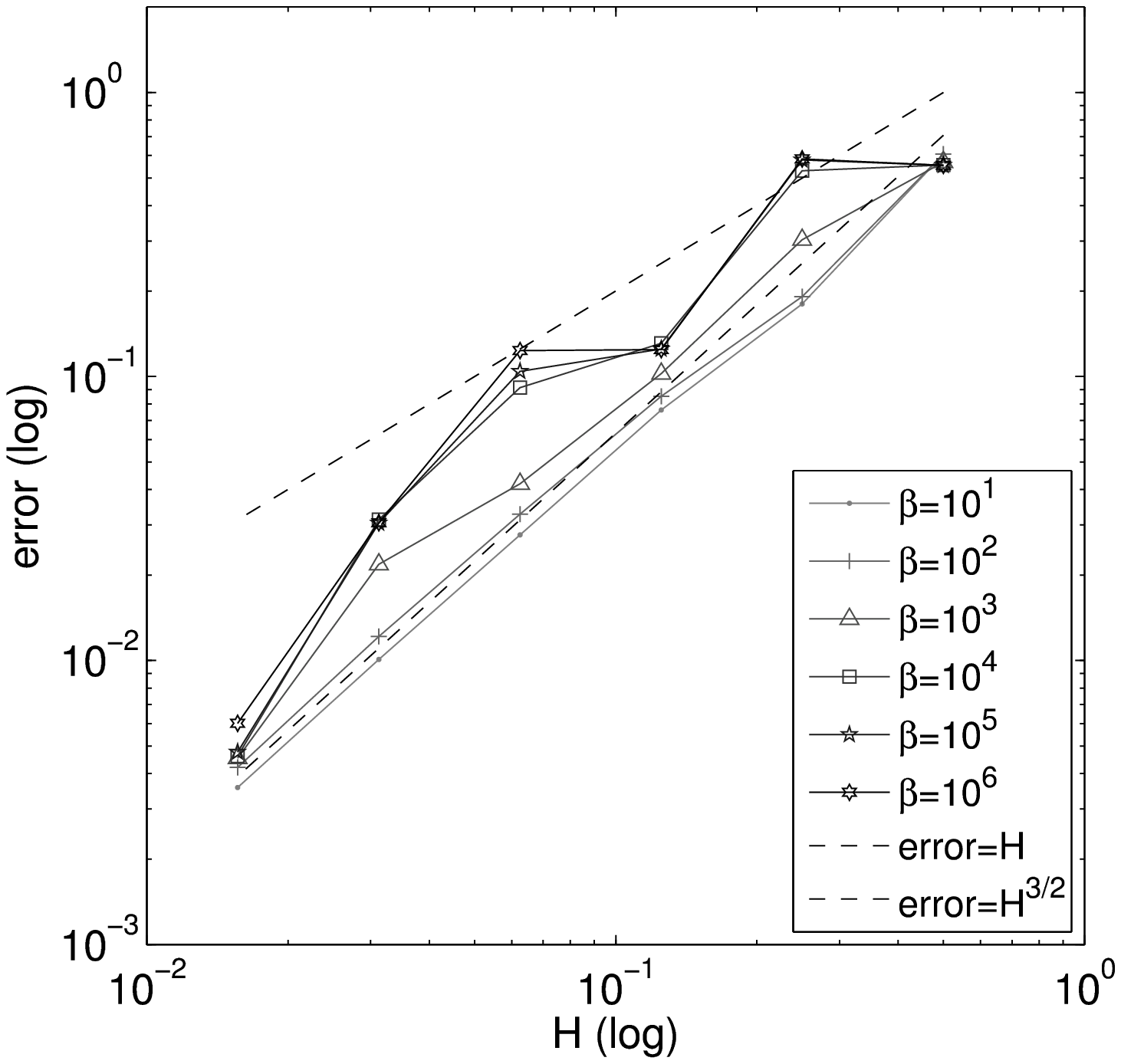}}
\subfigure[\label{fig:numexp2H_msfema}Results for $\QI^A$.
]{\includegraphics[width=0.38\textwidth]{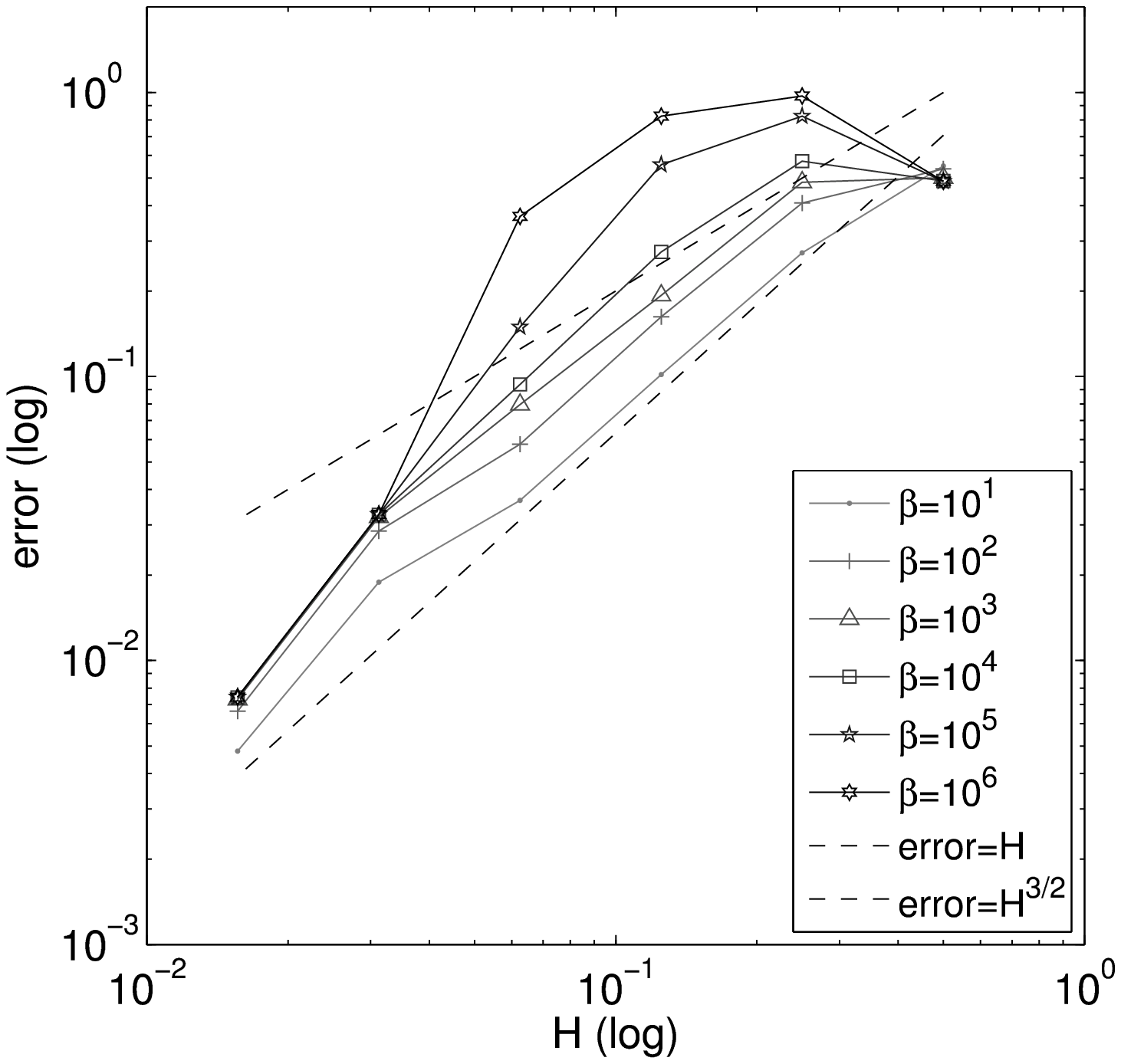}}\\
\subfigure[\label{fig:numexp2H_msfem1}Results for $\QI^{\operatorname{proj}}$.
]{\includegraphics[width=0.38\textwidth]{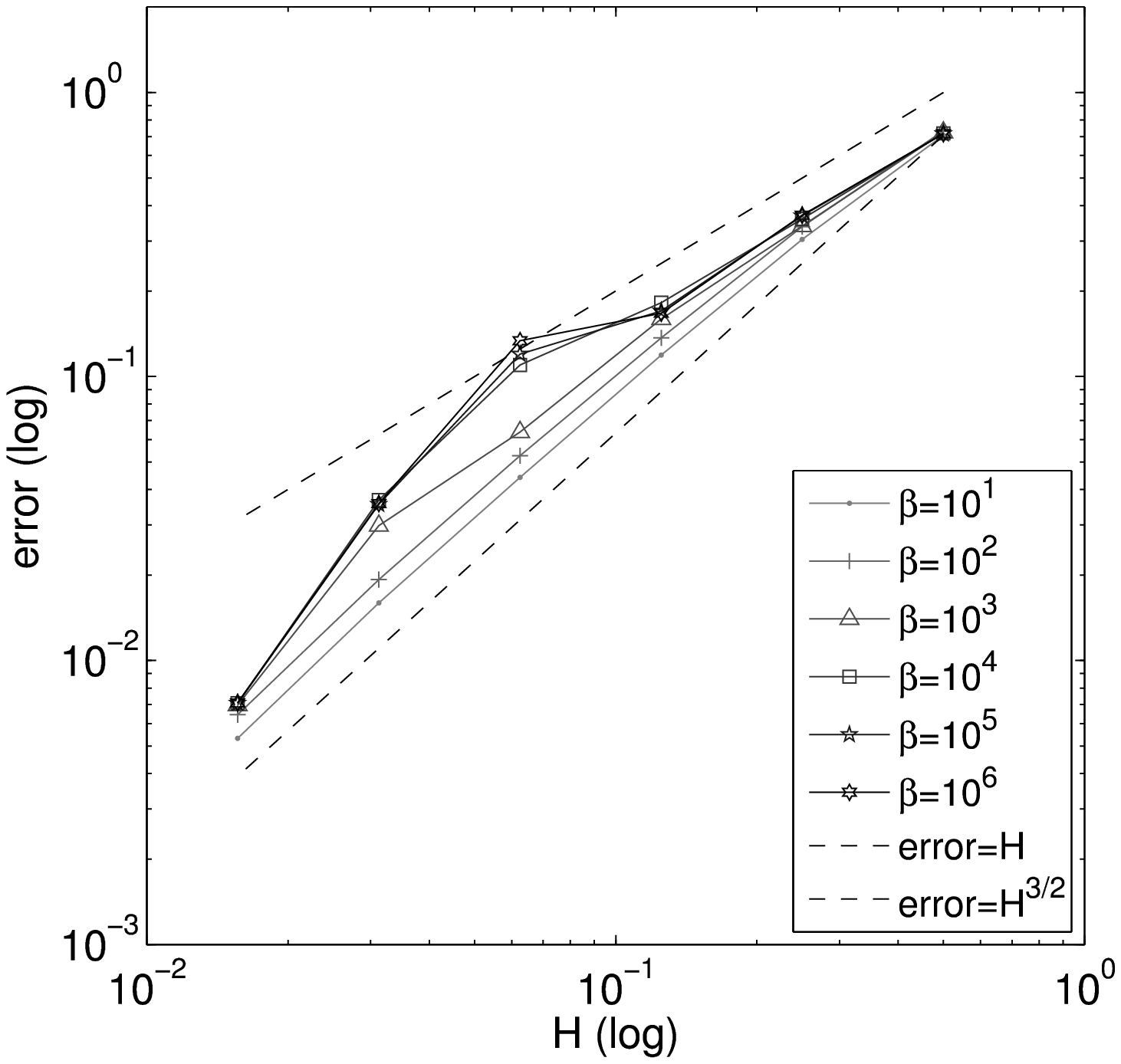}}
\subfigure[\label{fig:numexp2H_msfem1a}Results for $\QI^{\operatorname{proj},A}$.
]{\includegraphics[width=0.38\textwidth]{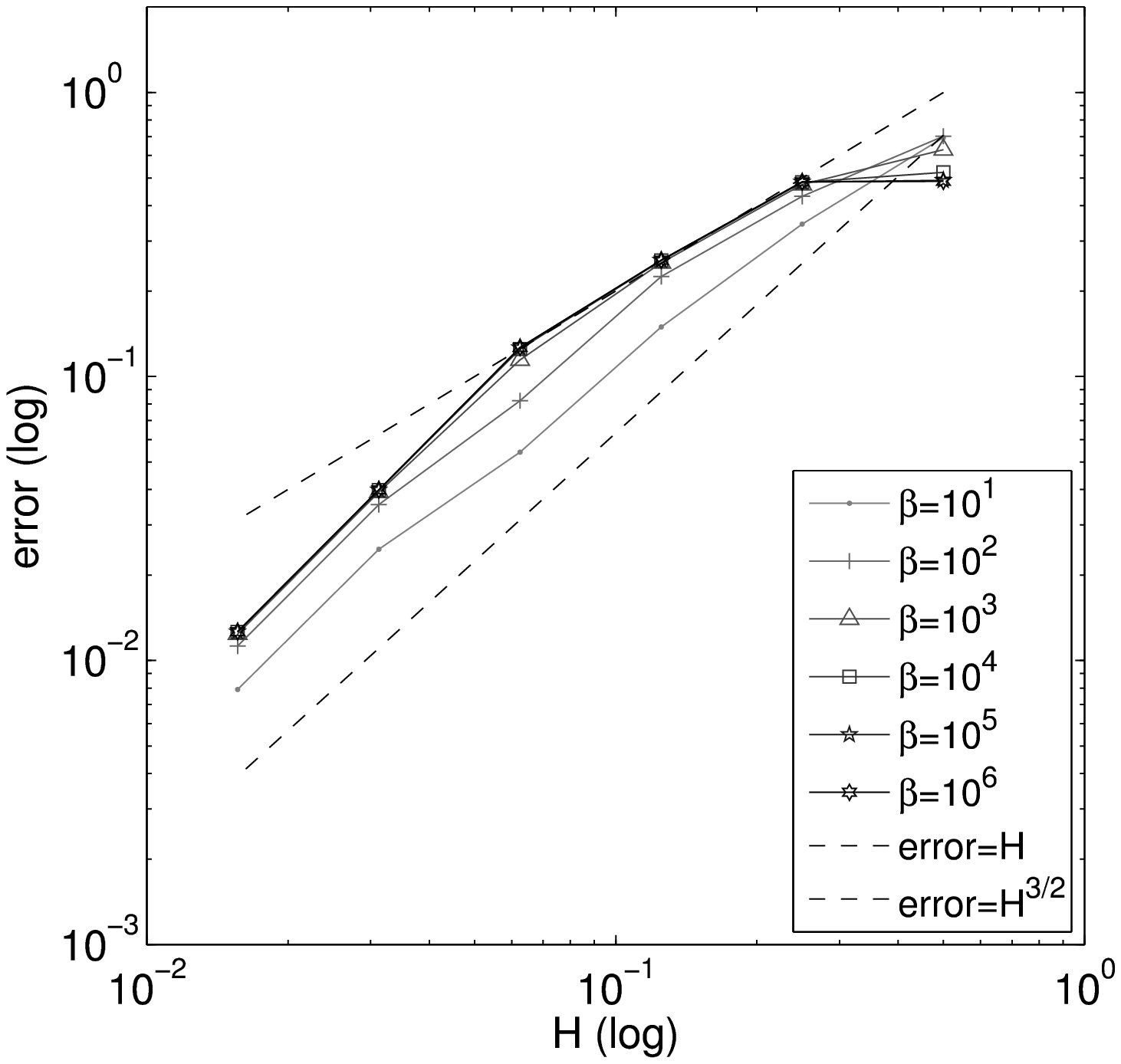}}\\
\end{center}
\caption{Numerical experiment of Section~\ref{ss:numexp2}: Results for high-contrast channels with several choices of the contrast parameter $\beta$ depending on the coarse mesh size $H$. The reference mesh size $h=2^{-8}$ remains fixed. The localization parameter $k=|\log_2 H|+1$ is tied to the coarse mesh size. \label{fig:numexp2H}}
\end{figure}

Apart from the coefficient, the experimental setup is exactly the same as in Section~\ref{ss:numexp1}. Figures~\ref{fig:numexp2H} and \ref{fig:numexp2k} show the results. The observations for the operators $\QI$, $\QI^{\operatorname{proj}}$, and $\QI^{\operatorname{proj},A}$ are similar as before. Again, the $A$-independent choices deliver more accuracy for sufficiently large localisation parameter whereas $\QI^{\operatorname{proj},A}$
is significantly more efficient for small $k$. 
By contrast, the operator $\QI^A$ performs much worse in this experiment. On the coarse meshes that do not resolve the coefficient, it requires a much larger choice of $k$ than the other operators to be accurate. We emphasize that this effect does neither contradict our theory nor can be explained by it. However, it clearly shows that the choice of the interpolation operator may have a large impact on the actual performance of the methods, a fact that motivates the further development and analysis of such operators.

\begin{figure}[tb]
\begin{center}
\subfigure[\label{fig:numexp2k_msfem}Results for $\QI$.
]{\includegraphics[width=0.38\textwidth]{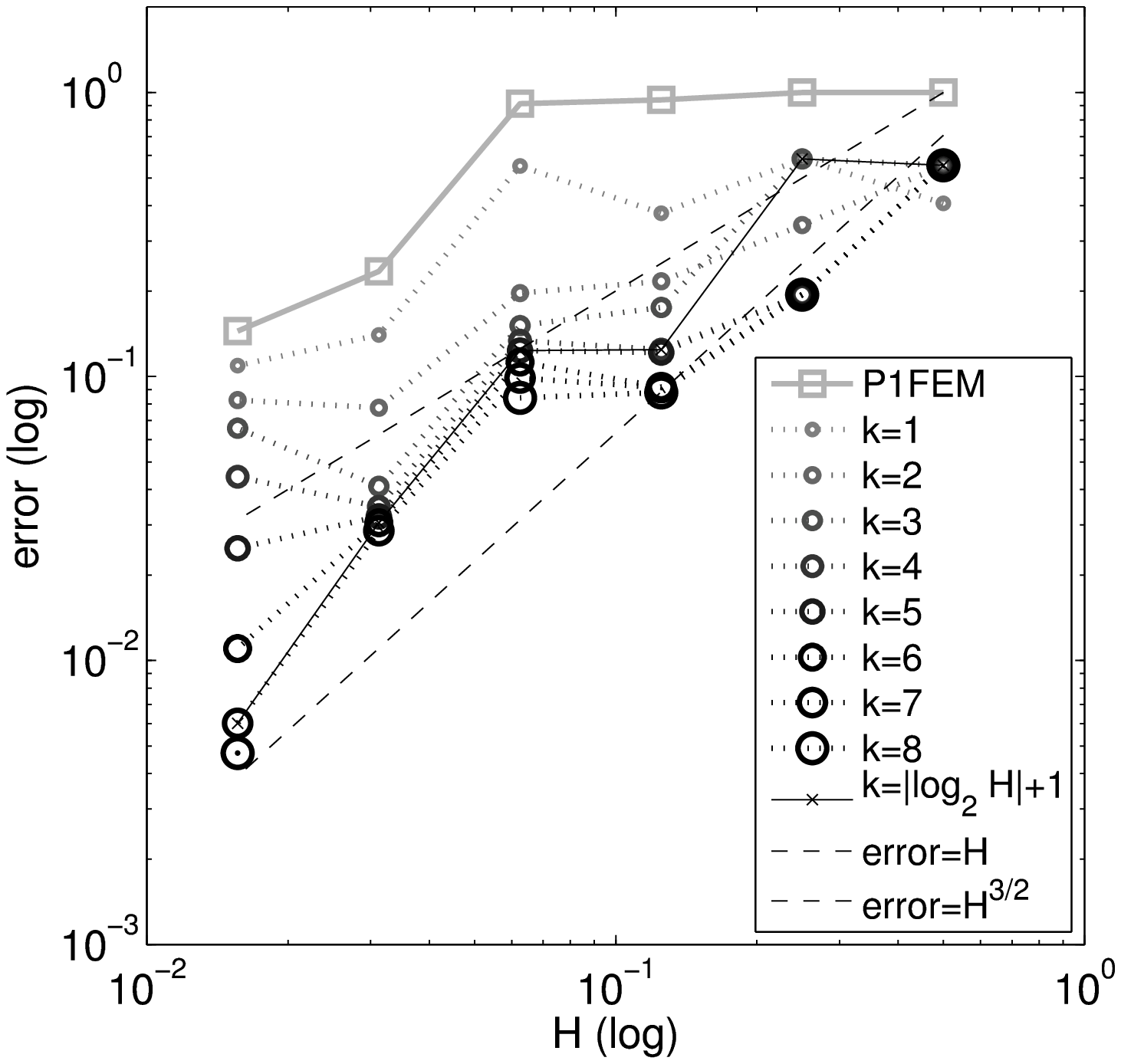}}
\subfigure[\label{fig:numexp2k_msfema}Results for $\QI^A$.
]{\includegraphics[width=0.38\textwidth]{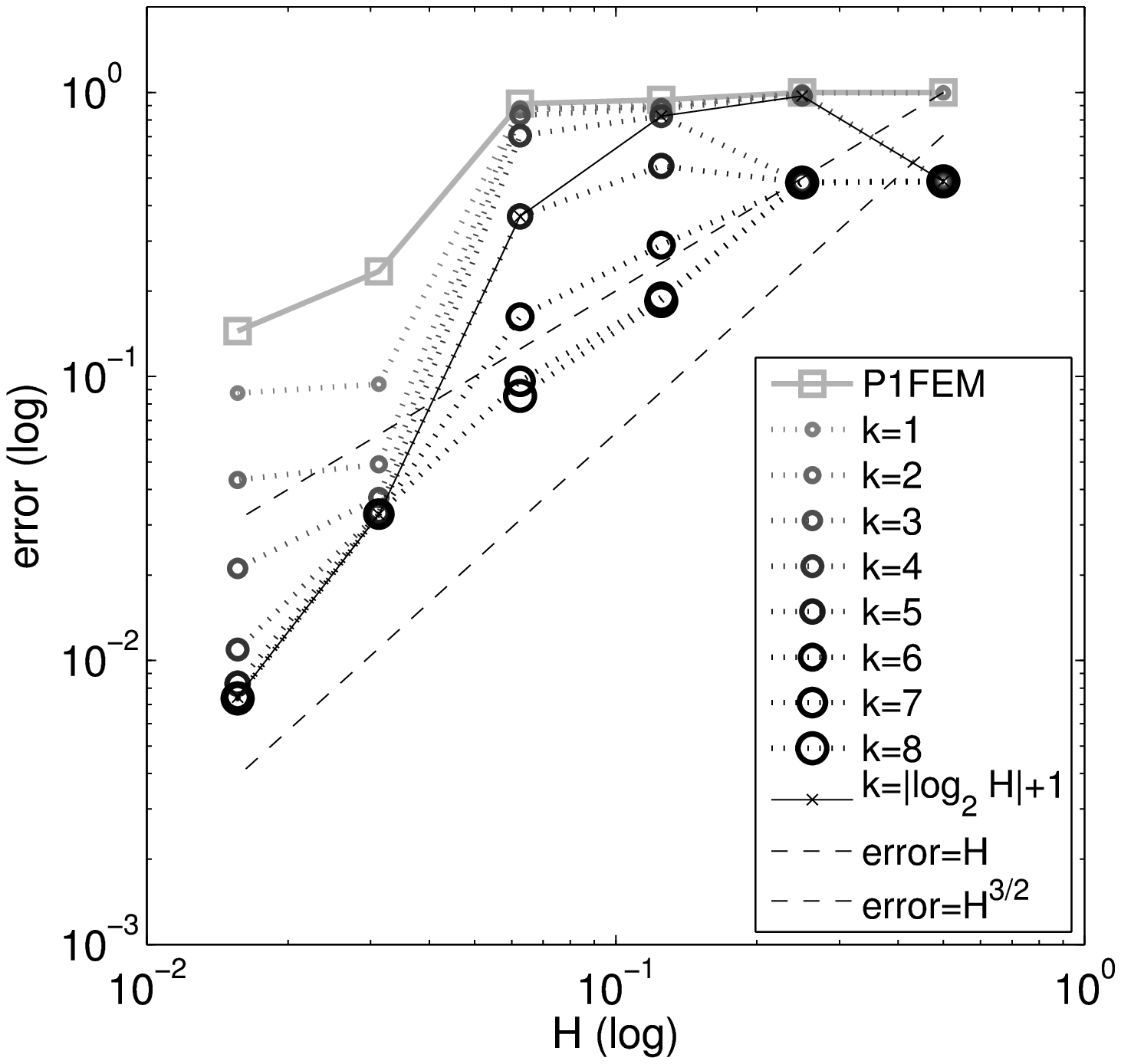}}\\
\subfigure[\label{fig:numexp2k_msfem1}Results for $\QI^{\operatorname{proj}}$.
]{\includegraphics[width=0.38\textwidth]{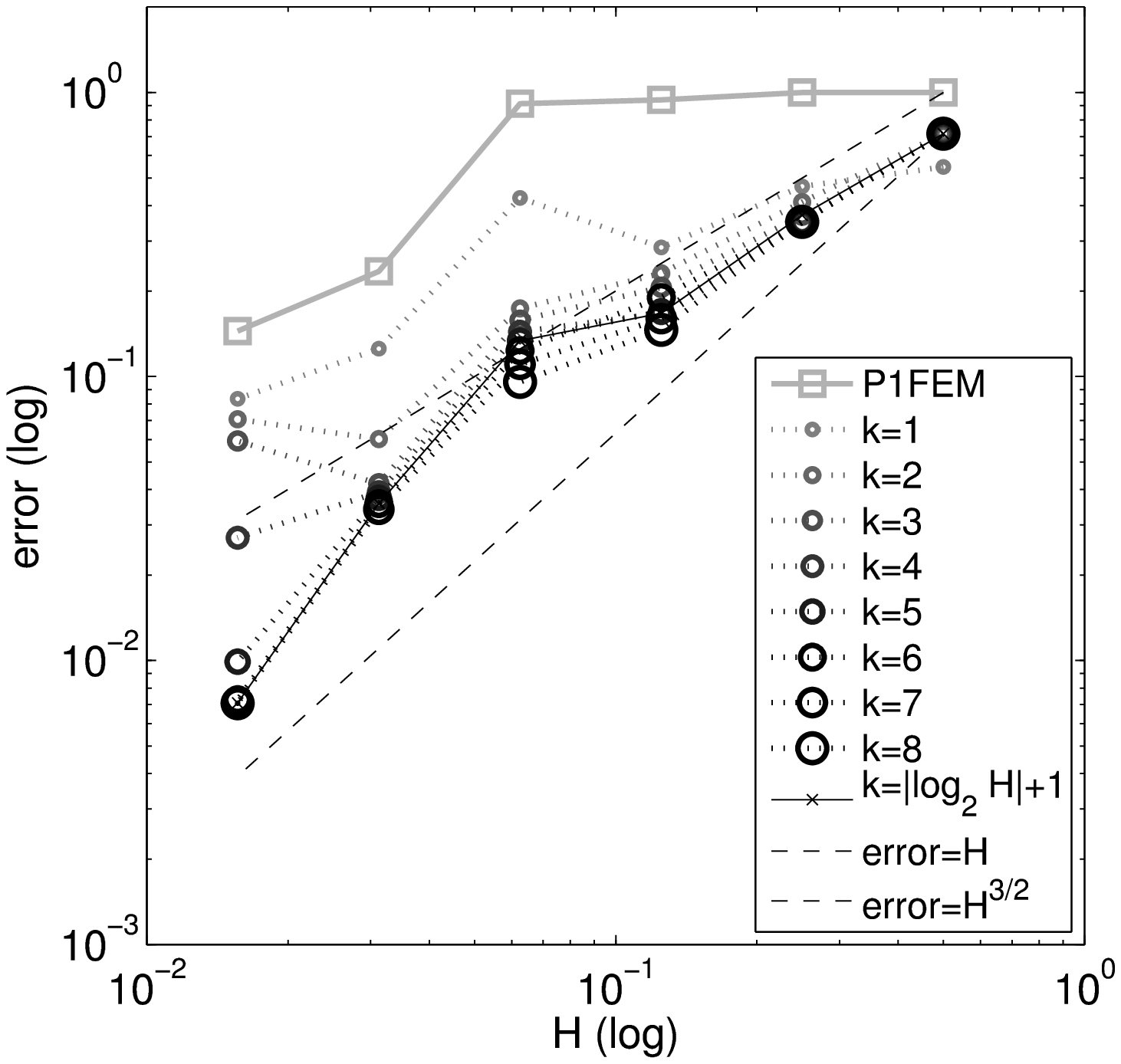}}
\subfigure[\label{fig:numexp2k_msfem1a}Results for $\QI^{\operatorname{proj},A}$.
]{\includegraphics[width=0.38\textwidth]{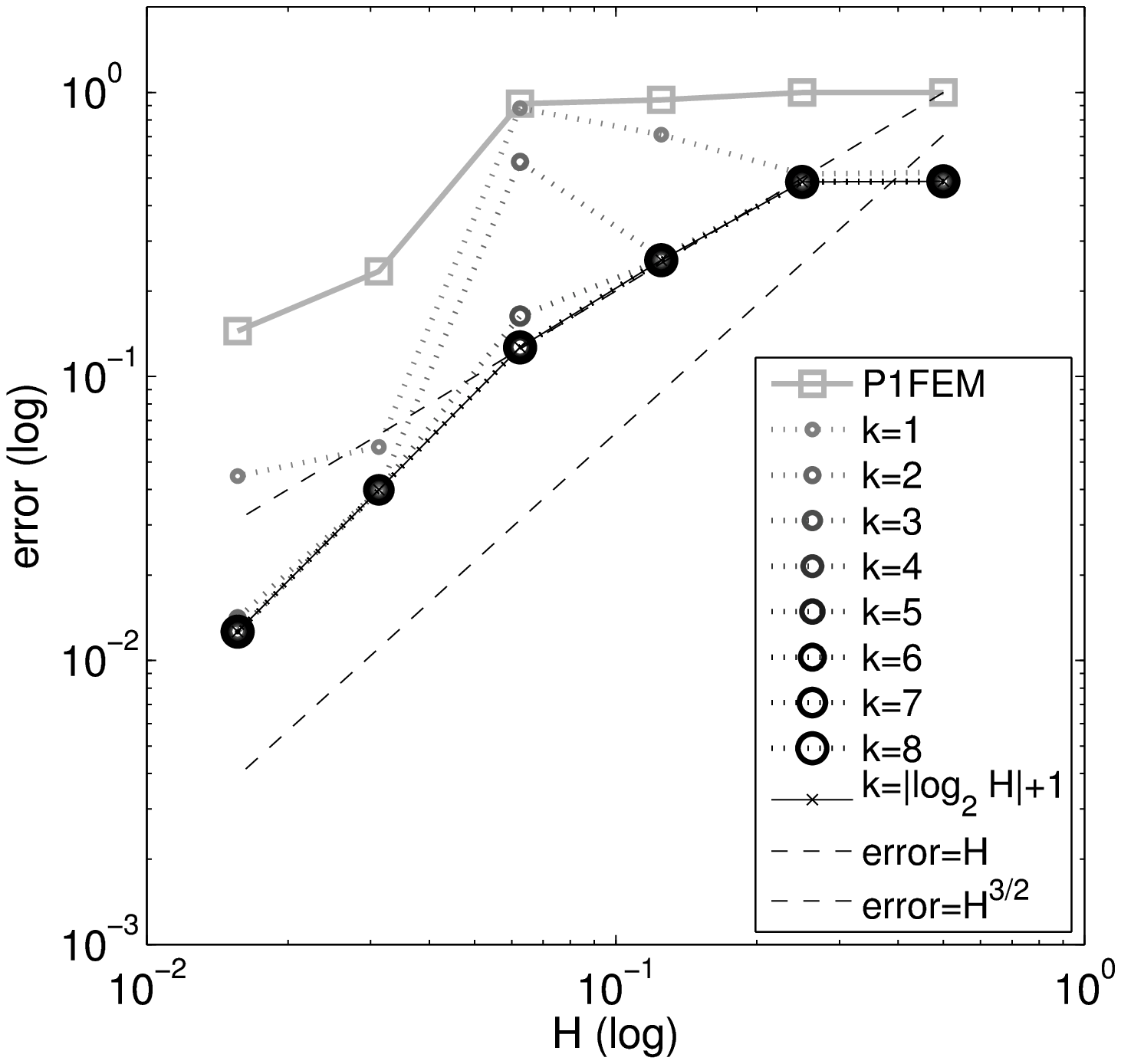}}\\
\end{center}
\caption{Numerical experiment of Section~\ref{ss:numexp2}: Results for high-contrast channels with contrast parameter $\beta$ depending on the coarse mesh size $H$. The reference mesh size $h=2^{-8}$ remains fixed. The localization parameter $k$ is varied between $1$ and $8$. \label{fig:numexp2k}}
\end{figure}

\subsection{Rough coefficient with multiscale features}\label{ss:numexp3}
\begin{figure}
\begin{center}
\includegraphics[width=0.45\textwidth]{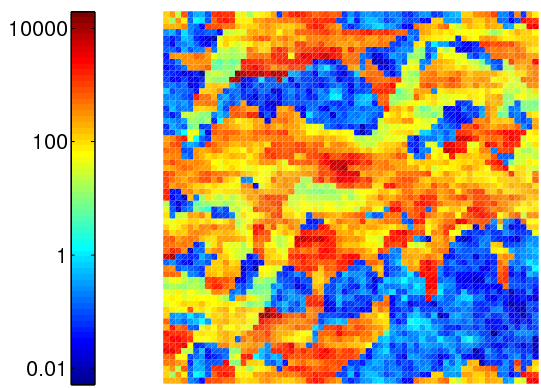}\end{center}
\caption{Scalar coefficient $A$ used in the numerical experiment of Section~\ref{ss:numexp3}.\label{fig:A2}}
\end{figure}
Let $\Omega:=(0,1)^2$ be the unit square. In this third experiment, the scalar coefficient $A$ (see Figure~\ref{fig:A2}) is piecewise constant with respect to a uniform Cartesian grid of width $2^{-6}$. Its values are taken from the data of the SPE10 benchmark, see \texttt{http://www.spe.org/web/csp/}. The coefficient is highly varying and strongly heterogeneous. The contrast for $A$ is large, $\beta/\alpha\approx 4\cdot 10^6$. This coefficient is certainly not quasi-monotone with regard to the coarse meshes considered here. The right-hand side term reads
\begin{equation*}
 g(x) = \begin{cases}
	  8,&x\in[0,\tfrac{1}{4}]\times[0,\tfrac{1}{4}]\cup[\tfrac{3}{4},1]\times[\tfrac{3}{4},1],\\
	  0,&\text{elsewhere.}
        \end{cases}
\end{equation*}
Consider uniform coarse meshes of size $\sqrt{2}H=2^{-1},2^{-2},\ldots,2^{-6}$ of $\Omega$ (cf. Figure~\ref{fig:meshes}). Note that none of these meshes resolves the rough coefficient $A$ appropriately. 
Again, the reference mesh $\tri_h$ has width $h=2^{-8}/\sqrt{2}$ and we compare the reference solution $u_h$ (with respect to the $P1$ conforming finite element approximation on the reference mesh $\tri_h$) with coarse scale approximations depending on the coarse mesh size $H$, several interpolation operators and the localization parameter $k$. Figure~\ref{fig:numexp3} depicts the results. 
\begin{figure}[tb]
\begin{center}
\subfigure[\label{fig:numexp3_msfem}Results for $\QI$.
]{\includegraphics[width=0.38\textwidth]{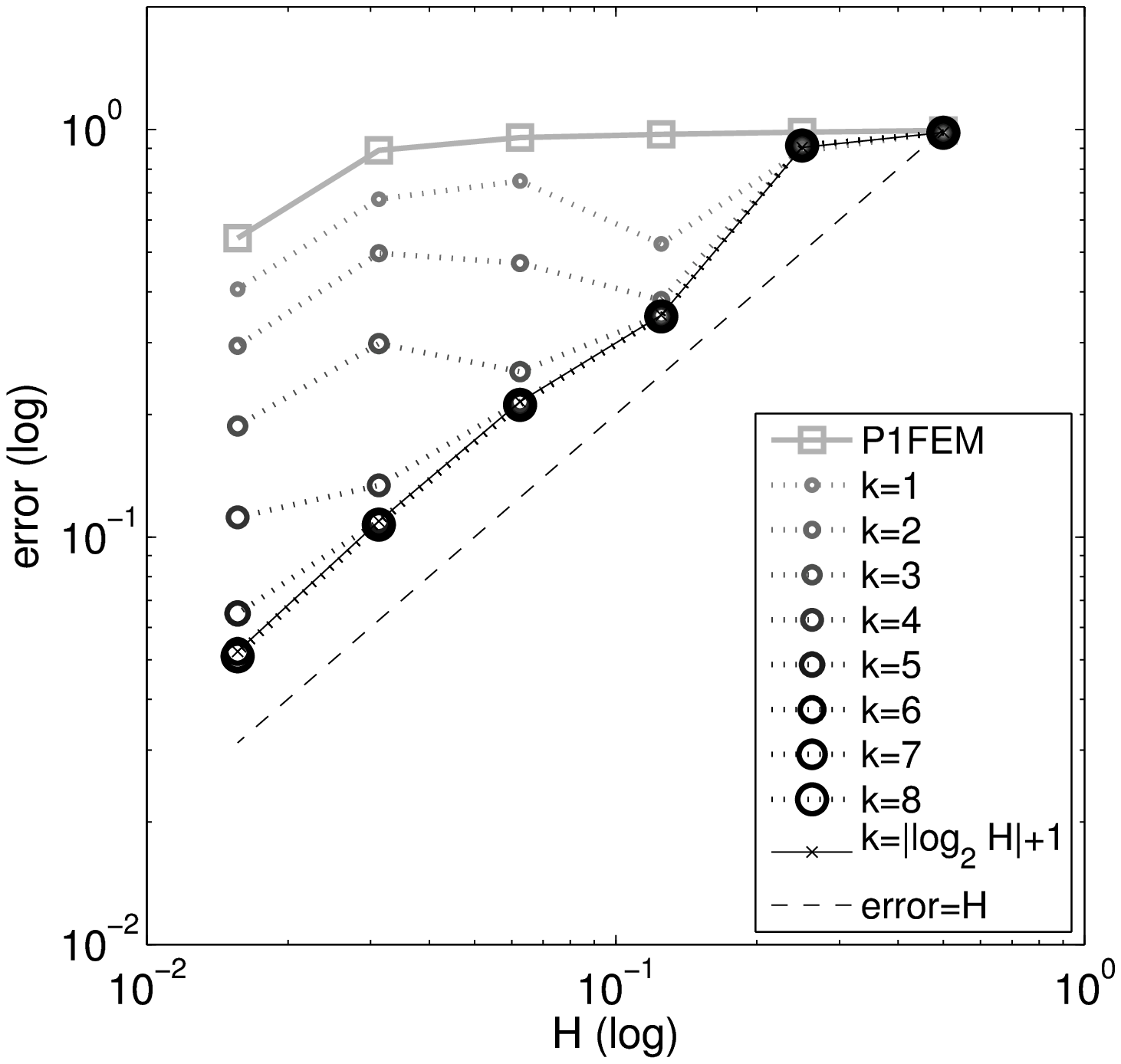}}
\subfigure[\label{fig:numexp3_msfema}Results for $\QI^A$.
]{\includegraphics[width=0.38\textwidth]{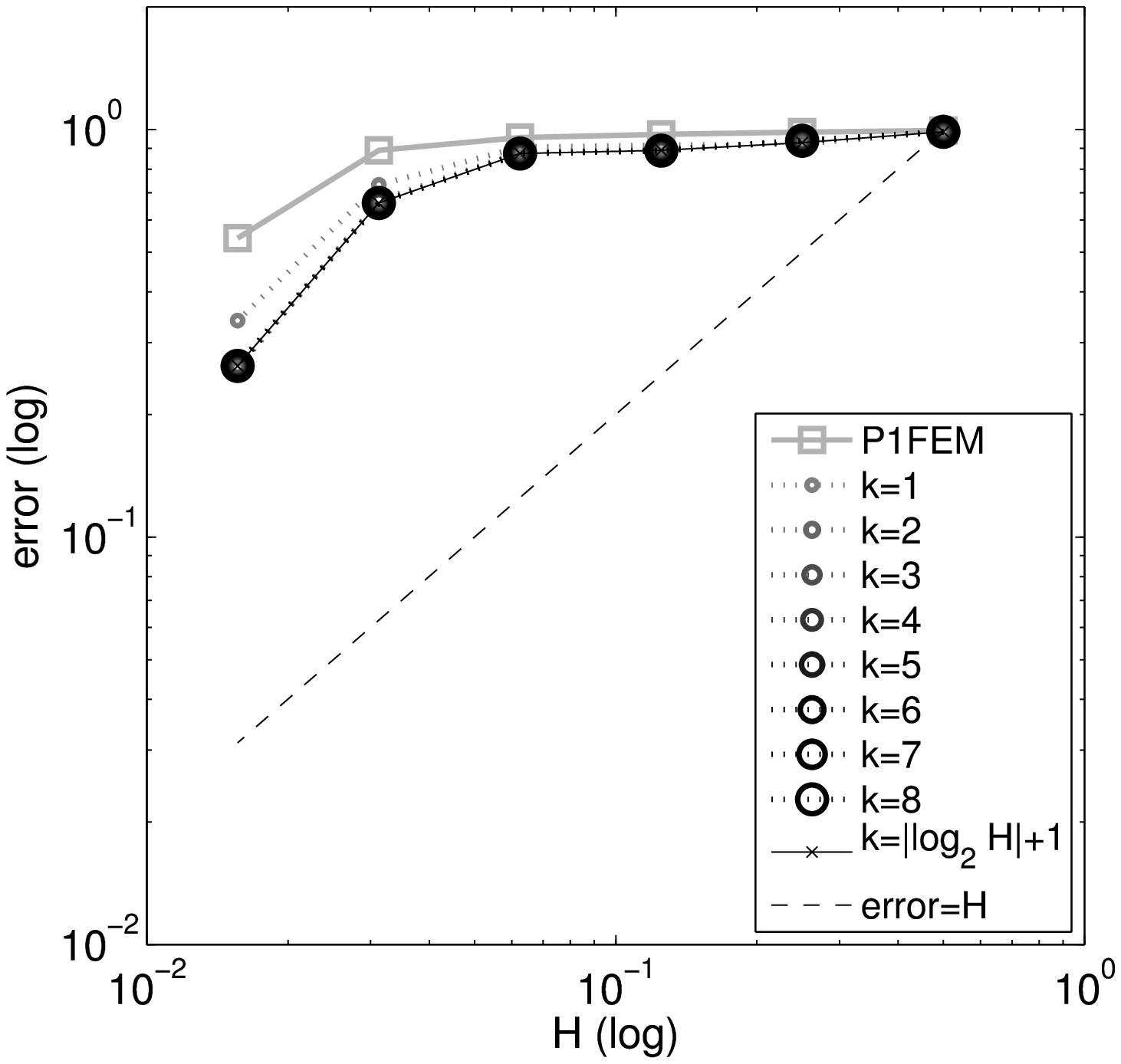}}\\
\subfigure[\label{fig:numexp3_msfem1}Results for $\QI^{\operatorname{proj}}$.
]{\includegraphics[width=0.38\textwidth]{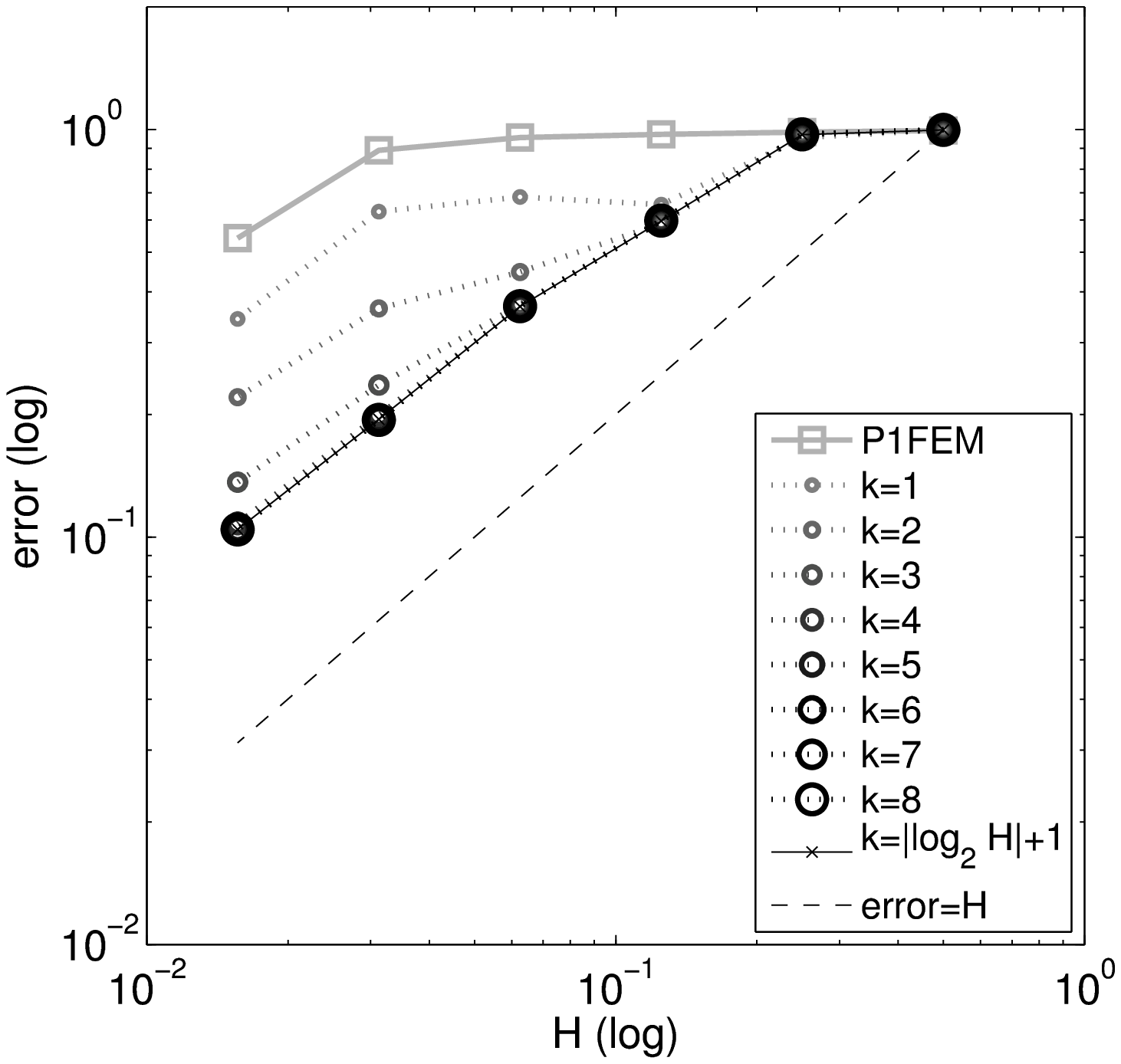}}
\subfigure[\label{fig:numexp3_msfem1a}Results for $\QI^{\operatorname{proj},A}$.
]{\includegraphics[width=0.38\textwidth]{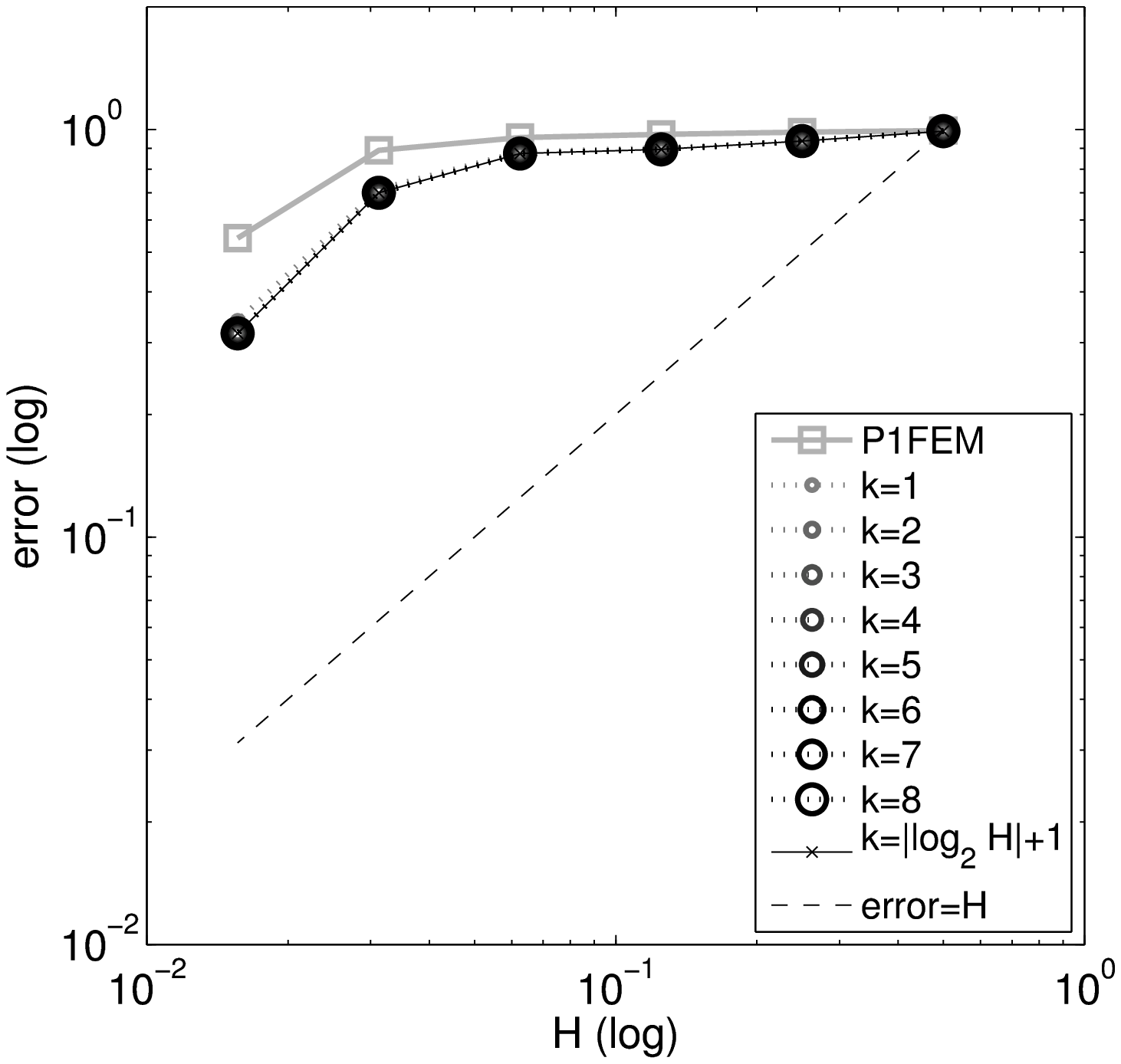}}\\
\end{center}
\caption{Numerical experiment of Section~\ref{ss:numexp3}: results for SPE10 data depending $H$. The reference mesh size $h=2^{-8}$ remains fixed. The localization parameter $k$ is varied between $1$ and $9$. \label{fig:numexp3}}
\end{figure}
This time, the methods based on $A$-independent interpolation perform significantly better that the methods with $A$-weighted interpolation. This superiority could be related to the approximability properties of the global bases. Note that, for non-quasi-monotone coefficients, the constant in Lemma~\ref{l:ideal} may depend on the contrast whereas the accuracy of the global method based on $\QI$ is independent of $\beta$ (cf. equation \eqref{e:estglobal}). Why this nice property of the $\QI$-based method is also observed after localization, however, remains completely open. 

To sum up, it can be said that the numerical experiments clearly showed the potential of the general methodology for high-contrast problems. They also showed that the decay of the correctors may be accelerated significantly by using $A$-dependent interpolation operators for the underlying split of coarse and fine scales in some cases. This is also supported by our theoretical results. However, the theory remains pessimistic in some cases and does not yet provide general advice regarding the choice of the interpolation operator along with an optimal choice of the localization parameter. \vspace{2ex}

\noindent
{\bf Acknowledgement.} We thank Clemens Pechstein for suggesting the alternative, projective  quasi-interpolation operator and providing us with the basic ideas for its analysis.

\end{document}